\documentclass{amsart}
\usepackage{amscd}
\usepackage{amssymb}
\begin{document}
\title{Unstable operations in \'etale and motivic cohomology}
\date{\today}
\author{Bert Guillou}
\address{Dept.\ of Mathematics, University of Kentucky} 
\email{bertguillou@uky.edu} 
\author{Chuck Weibel}
\address{Dept.\ of Mathematics, Rutgers University, New Brunswick,
NJ 08901, USA} \email{weibel@math.rutgers.edu}
\newcommand{\edit}[1]{}
\def\map#1{{\buildrel #1 \over \longrightarrow}}
\def\lmap#1{{\buildrel #1 \over \longleftarrow}}
\newcommand{\mathdot}{{\mathbf{\scriptscriptstyle\bullet}}}
\def\oo{\otimes}
\def\tr{{\text{tr}}}
\def\F{\mathbb F_\ell}\def\Z{\mathbb Z}
\def\Ftr{\mathbb F_{\ell,\tr}}
\def\muell#1{\mu_\ell^{\otimes #1}}
\def\cA{\mathcal A}
\def\cF{\mathcal F}
\def\cO{\mathcal O}
\def\ie{{\em i.e., }}
\def\et{{\text{et}}}\def\nis{{\text{nis}}}
\def\top{{\text{top}}}
\def\Hom{\operatorname{Hom}}   \def\Ext{\operatorname{Ext}}
\def\End{\operatorname{End}}
\def\Tot{\operatorname{Tot}} 
\def\Spec{\operatorname{Spec}} 
\def\into{\hookrightarrow}
\renewcommand{\choose}[2]{\genfrac{(}{)}{0pt}{}{#1}{#2}}

\numberwithin{equation}{section}
\newtheorem{thm}[equation]{Theorem}
\newtheorem{cor}[equation]{Corollary}
\newtheorem{prop}[equation]{Proposition}
\newtheorem{lem}[equation]{Lemma}
\newtheorem*{theorem}{equation}
\newtheorem{Kudosthm}[equation]{Kudo's Theorem}
\newtheorem{NormResidueThm}[equation]{Norm Residue Theorem}
\theoremstyle{definition}
\newtheorem{defn}[equation]{Definition}
\newtheorem{ex}[equation]{Example}
\newtheorem{rem}[equation]{Remark}
\newtheorem*{remm}{Remark}
\newtheorem{substuff}{\bf Remark}[equation]
\newtheorem{subrem}[substuff]{\bf Remark} 
\newtheorem{subex}[substuff]{\bf Example}
\newtheorem{subcor}[substuff]{\bf Corollary}
\newtheorem{bistable}[equation]{Bistable Operations}
\newtheorem{construction}[equation]{Construction}
\newtheorem{Conjecture}[equation]{Conjecture}
\begin{abstract}
We classify all \'etale cohomology operations on $H_\et^n(-,\muell{i})$,
showing that they were all constructed by Epstein. We also construct
operations $P^a$ on the mod-$\ell$ motivic cohomology groups $H^{p,q}$, 
differing from Voevodsky's operations;
we use them to classify all motivic cohomology operations on 
$H^{p,1}$ and $H^{1,q}$ and suggest a general classification. 
\end{abstract}
\maketitle

In the last decade, several papers have given constructions of
\edit{Intro redone 12/28}
cohomology operations on motivic and \'etale cohomology, following the
earlier work of Jardine \cite{Jardine}, Kriz-May \cite{KM} and 
Voevodsky \cite{V96, RPO}: see \cite{BJ, BJ1, Joshua, May1, V11, V-MC/l}.  
The goal of this paper is to provide, for each $n$ and $i$, 
a classification of all such operations on the \'etale groups
$H_\et^n(-,\muell{i})$ and the motivic groups $H^{n,i}(-,\F)$, 
similar to Cartan's classification of operations on singular cohomology 
$H_\top^n(-,\F)$ in \cite{Cartan}. We succeed for 
\'etale operations and partially succeed for motivic operations.

We work over a fixed field $k$ and fix a prime $\ell$ with $1/\ell\in k$.
By definition, an (unstable) \'etale cohomology operation on 
$H_\et^n(-,\muell{i})$ is a natural transformation 
$H_\et^n(-,\muell{i})\to H_\et^p(-,\muell{q})$
of set-valued functors from the category of (smooth) simplicial schemes
over $k$ (for some $p$ and $q$). 
Similarly, a motivic cohomology operation on $H^{n,i}$ is a
natural transformation $H^{n,i}\to H^{p,q}$ on this category.
By definition, $H^{p,q}(X)$ denotes the Nisnevich cohomology
$H_\nis^p(X,\F(q))$, where the cochain complex $\F(q)$ is defined 
in \cite{V96} or \cite{MVW}. Note that the set of all cohomology
operations forms a ring; the product of $\theta_1$ and $\theta_2$ 
is the operation $x\mapsto\theta_1(x)\cdot\theta_2(x)$.

Our classification begins with a construction of \'etale operations $P^a$,
due to Epstein, as a special case of the operations on sheaf cohomology 
he described in his 1966 paper \cite{Epstein}.
This is carried out in Theorem \ref{etale-exist} for odd $\ell$; 
when $\ell=2$, this was established by Jardine \cite{Jardine}.
The coefficient ring $H_\et^*(k,\muell{*})$ also acts on \'etale
cohomology; we prove in Theorem \ref{thm:etaleops} that the ring of
all \'etale operations on $H_\et^n(-,\muell{i})$ is the twisted ring
$H_\et^*(k,\muell{*})\oo H_\top^*(K_n)$, where $H_\top^*(K_n)$ 
is Cartan's ring of operations on $H_\top^n(-,\F)$. 
(We give a precise description of Cartan's ring in Definition \ref{def:PI}
below.) This classifies all \'etale operations on $H_\et^n(-,\muell{i})$:
they are $H_\et^*(k,\muell{*})$-linear combinations of monomials in the
operations $P^I$.


A slightly different approach to cohomology operations was given by 
May in \cite{May}, one which produces operations in the cohomology of any 
\edit{redone 12/31}
$E_\infty$-algebra. In Section \ref{sec:theta}, we review this approach 
in the context of sheaf cohomology and show in Corollary \ref{PQagree}
that the \'etale operations constructed in this way agree with Epstein's. 
This result allows us to utilize May's treatment of the Kudo Transgression 
Theorem in \cite[3.4]{May}; see Theorem \ref{Kudo} below.

In Section \ref{sec:Amotivic}, we combine Epstein's construction 
with the Norm Residue Theorem 
to define motivic operations $P^a$ (see \ref{motivic-exist}). 
We show they are compatible with the \'etale operations,
and that they are stable under simplicial suspension.
The operation $P^0$ is the Frobenius $H^{n,i}\to H^{n,i\ell}$ on
motivic cohomology, induced by the $\ell^{th}$ power map
$\F(i)\to\F(i\ell)$; see Proposition \ref{P0=Frob}. 
%
One new result concerning Voevodsky's operations is that
for $n>i$ and $x\in H^{2n,i}$ we have 
$P_V^n(x)=[\zeta]^{(n-i)(\ell-1)}x^\ell$ (see Corollary \ref{PVn}).
This extends Lemma 9.8 of \cite{RPO}, which states that
$P^n(x)=x^\ell$ for $x\in H^{2n,n}(X)$.

The classification of motivic cohomology operations is complicated
by the presence of more operations than those constructed by 
Voevodsky or via Steenrod-Epstein methods. One example is that 
an $\ell$-torsion element $t$ in the Brauer group of $k$
gives an operation $H^{1,2}\to H^{3,3}$ by
\[ 
H^{1,2}(X) \cong H^1_\et(X,\muell{2}) \map{\cup t}
H^3_\et(X,\muell{3}) \cong H^{3,3}(X).
\]
Also unexpectedly, we may also use $t$ and the Bockstein $\beta$ to get 
an operation $H^{1,2}(X)\to H^{4,3}(X)$ (see Example \ref{ex:b=1} below).
When $k$ contains a primitive $\ell^{th}$ root of unity $\zeta$,
we also have an interesting operation $H^{1,2}(X) \to
H^{2,1}(X)=\text{Pic}(X)/\ell$:  divide by the Bott element
$[\zeta]\in H^{0,1}(k)$ and then apply the Bockstein;
see Proposition \ref{H1i_withzeta}.

In Section \ref{sec:weight1}, we determine the ring of all unstable
motivic cohomology operations on $H^{n,1}$. If $\ell\ne2$, it is 
the twisted ring $H^{*,*}(k)\oo H_\top^*(K_n)$, where $H^{*,*}(k)$
is the motivic cohomology of $k$ and $H_\top^*(K_n)$ is Cartan's
ring, described in Definition \ref{def:PI} below.
%


In Section \ref{sec:degree1}, we determine the ring of unstable
cohomology operations on $H^{1,i}$. When $k$ contains the $\ell^{th}$ 
roots of unity,
this is the graded polynomial ring over $H^{*,*}(k)$ on operations 
$\gamma: H^{1,i}(X)\cong H^{1,1}(X)$ and its Bockstein, where
$\gamma$ is given by the Norm Residue Theorem \ref{NRT}. 
For general fields, it is the Galois-invariant subring.
The operations on $H^{1,2}$ referred to above arise in this way.

Finally Section \ref{sec:conjectures} contains a conjecture about
what the general classification might be for $H^{n,i}$ when $n,i>1$.

\medskip
Since it is the topological prototype of our classification theorem,
we conclude this introduction with a description of
the ring of all singular cohomology operations on $H_\top^n(-,\F)$.
Serre observed that the ring of operations from $H_\top^n(-,\F)$ 
to $H_\top^*(-,\F)$ is isomorphic to the cohomology 
$H_\top^*(K_n)$ of the Eilenberg-Mac\,Lane space $K_n=K(\F,n)$;
the structure of this ring was determined by 
Serre and Cartan in \cite{Cartan} \cite{Cartan54}.
The following description is taken from \cite[6.19]{McCleary}.

\begin{defn}\label{def:PI}
For $\ell>2$,
let $H_\top^*(K_n)$ denote the free graded-commutative $\F$-algebra 
generated by the elements $P^I(\iota_n)$,
where $I=(\epsilon_0,s_1,\epsilon_1,...,s_k,\epsilon_k)$ is an admissible 
sequence satisfying either $e(I)<n$ or $e(I)=n$ and $\epsilon_0=1$.

Here the {\it excess} of $I$ is defined to be
$e(I)=2\sum(s_i-\ell s_{i+1}-\epsilon_i) + \sum_{i=0}^k \epsilon_i$,
where $s_i=0$ for $i>k$,
and $I$ is {\it admissible} if $s_i\ge\ell s_{i+1}+\epsilon_i$ for all $i<k$.

When $\ell=2$, $H_\top^*(K_n)$ denotes the free graded-commutative 
$\mathbb F_2$-algebra generated by the elements $Sq^I(\iota_n)$,
with $I=(s_1,...,s_k)$ admissible ($s_i\ge2s_{i+1}$) and $e(I)<n$,
where the excess is $e(I)=\sum(s_i-2s_{i+1})=s_1-\sum_{i>1}s_i$.
\end{defn}

For example, every operation on $H_\top^2(-,\F)$ is a polynomial in 
$\text{id}$, $\beta$, the $P^I\beta$ and the $\beta P^I\beta$
(where $P^I=P^{\ell^k}\!\cdots P^\ell P^1$). This is because
the only admissible sequences with excess $<2$ are
$0$, $(1)$ and $(0,\ell^k,0,\dots,\ell,0,1,1)$.

\newpage
\section{Epstein's \'etale construction}

Cohomology operations in \'etale cohomology were constructed by
D.\,Epstein long ago in the 1966 paper \cite{Epstein}, and 
(for constant coefficients) made explicit by M.\,Raynaud \cite[4.4]{Raynaud}. 
Alternative constructions were later given by 
L.\,Breen \cite[III.4]{Breen} and J.F.\,Jardine \cite[1.4]{Jardine},
\cite[\S2]{Jardine-Spinor}.

In Epstein's approach, one starts with an $\F$-linear tensor abelian category 
$\cA$ (such as sheaves of $\F$-modules on a site),
a left exact functor $H^0(X,-)$ (global sections over $X$) and a commutative
associative ring object $\cO$ of $\cA$. Epstein constructs operations 
$Sq^a\!:H^n(X,\cO)\to H^{n+a}(X,\cO)$ if $\ell=2$, and
\[ 
P^a:H^n(X,\cO)\to H^{n+2a(\ell-1)}(X,\cO), \quad \ell\ne2, 
\]
satisfying the usual relations:
$P^ax=x^\ell$ if $n=2a$, $P^ax=0$ if $n<2a$, a Cartan relation for
$P^a(xy)$ and Adem relations for $P^aP^b$. Epstein also defines an
operation $Q^a$ for each $a$, whose degree is one more than that of $P^a$. 
One subtlety is that $P^0$ is not the
identity but rather the Frobenius map on the $\F$-algebra $H^0(X,\cO)$.

Now suppose that $\cA$ is the category of \'etale sheaves (on the big
\'etale site) and that
$\cO$ is the graded \'etale sheaf $\oplus_{i=0}^\infty \muell{i}$. 
We will prove in \ref{Bock-etale} below that $Q^a=\beta P^a$.
With this dictionary, Epstein's theorem specializes to yield

\begin{thm}\label{etale-exist}
For each odd prime $\ell$, there are additive cohomology operations
\[
P^a:H_\et^n(X,\muell{i})\to 
    H_\et^{n+2a(\ell-1)}(X,\muell{i\ell})
\]
and a Bockstein $\beta: H_\et^n(X,\muell{i})\to 
H_\et^{n+1}(X,\muell{i})$ satisfying the usual relations:
$P^ax\!=x^\ell$ if $n=2a$, $P^ax\!=0$ if $n<2a$, the Cartan relation
$P^a(xy)=\sum P^i(x)P^j(y)$ and Adem relations for 
both $P^aP^b$ ($a<b\ell$) and $P^a\beta P^b$ ($a\le b\ell$). 

When $\ell=2$, there are Steenrod operations 
$Sq^a: H_\et^n(X,\mu_2^{\oo i})\to H_\et^{n+a}(X,\mu_2^{\oo 2i})$,
or $H_\et^n(X,\Z/2)\to H_\et^{n+a}(X,\Z/2)$,
satisfying the usual relations. 
\end{thm}

\begin{proof}
The existence and basic properties is given in Chapter~7 of \cite{Epstein};
The Adem relations are established in \cite[9.7--8]{Epstein}, using
the dictionary that $P^a\beta P^b=P^aQ^b$.
\end{proof}

Note that, although the operations multiply the weight by $\ell$,
the reindexing makes no practical difference because 
there are canonical isomorphisms $\mu_\ell\cong\muell{\ell}$
and $\muell{i}\cong\muell{i\ell}$. We have emphasized the twist because
of our application to motivic operations below.

\begin{subrem}\label{natural-etale}
The operations $P^a$ are natural in $X$: if $f:X\to Y$ is a morphism
of simplicial schemes then $f^*P^a=P^a f^*$. This is immediate from
the naturality of the construction of $P^a$ with respect to the left
exact functor $H_\et^0(X,-)$, and also follows from \cite[11.1(8)]{Epstein}.

If $Z$ is a closed simplicial subscheme of $X$, we get 
cohomology operations $P^a$ on the relative groups 
$H_\et^n(X,Z;\muell{i})$, by replacing $H^0(X,-)$ by the left exact
functor $H_\et^0(X,Z;-)$. The same argument shows that $P^a$ is natural
in the pair $(X,Z)$.
\end{subrem}

\smallskip
For later use, we reproduce two key results from \cite{Epstein}.
If $\pi$ is a finite group, we write $\cA[\pi]$ for the category of
$\pi$-equivariant objects of $\cA$, i.e., objects $A$ equipped with a 
homomorphism $\pi\to\End(A)$. If $A$ is in $\cA[\pi]$ then $H^0(X,A)$
is a $\pi$-module, and we define the left exact functor $H_\pi^0(X,-)$ 
on the category $\cA[\pi]$ by the formula $H_\pi^0(X,A)=H^0(X,A)^\pi$.
We write $H_\pi^*(X,-)$ for the derived functors of $H_\pi^0(X,-)$.

\begin{thm}\label{Kunneth} 
Let $A$ be a bounded below cochain complex of objects of $\cA$, 
on which $\pi$ acts trivially. Then there is a natural isomorphism
\[ H^*(\pi,\F)\oo H^*(X,A) \map{\simeq} H_\pi^*(X,A). \]
\end{thm}

\begin{proof} (See \cite[4.4.4]{Epstein}.)
Let $C_*\!\to\Z$ be the standard periodic $\Z[\pi]$-resolution
\cite[6.2.1]{WHomo}, with generator $e_k$ of $C_k\cong\Z[\pi]$,
and set $C^*\!=\Hom(C_*,\F)$; thus
$H^*(\pi,\F)$ is the cohomology of $(C^*)^\pi$.
Choose a quasi-isomorphism $A\map{\sim} I^*$ with the $I^i$ injective in $\cA$.
Since $\pi$ acts trivially on $A$, we have quasi-isomorphisms of complexes 
in $\cA[\pi]$:  $A\to I^*= \Z\otimes I^* \to \Tot(C^*\otimes I^*)$.
Since each $C^n$ is a free $\F[\pi]$-module of finite rank,
$C^n\otimes I^q\cong \Hom(C_n,I^q)$ is injective in $\cA[\pi]$. 
Thus $\Tot(C^*\otimes I^*)$ is an injective replacement for $A$ in $\cA[\pi]$.

By definition, $H_\pi^*(X,A)$ is the cohomology of the total complex of 
\[
H_\pi^0(X,C^*\otimes I^*)=(C^*)^\pi \otimes I^*(X).
\]
The K\"unneth formula tells us this is the tensor product of 
the cohomology of $(C^*)^\pi$ and $I^*(X)$, \ie of
$H^*(\pi,\F)$ and $H^*(X,A)$.
\end{proof}

\begin{cor}
If $A$ is a commutative algebra object, the isomorphism of Theorem
\ref{Kunneth} is an algebra isomorphism.
\end{cor}

\begin{proof}
The verification that a commutative associative product on $A$ induces
an algebra structure on $H^*(X,A)$ and $H_\pi^*(X,A)$ is well known.
We omit the standard proof that the isomorphism above commutes with products.
\end{proof}

Recall that for any complex $C$, the symmetric group $S_\ell$ acts on 
$C^{\otimes\ell}$ by permuting factors with the usual sign change.
Now suppose that $\pi$ is the cyclic Sylow $\ell$-subgroup of $S_\ell$.
Choosing an injective replacement $A^{\otimes\ell}\to J^*$ in $\cA[\pi]$, 
the comparison theorem \cite[2.3.7]{WHomo} lifts the equivariant 
quasi-isomorphism $A^{\oo\ell}\to(I^*)^{\oo\ell}$ to an equivariant map 
$(I^*)^{\otimes\ell}\to J^*$, unique up to chain homotopy.

Since $H^*(X,A)$ is the cohomology of $I^*(X)$, we can represent
any element of $H^n(X,A)$ by an $n$-cocycle $u\in I^n(X)$. 
The $n\ell$-cocycle $u\otimes\cdots\otimes u$ of $I(X)^{\otimes \ell}$
is $\pi$-invariant, because the generator of $\pi$ acts as multiplication
by $(-1)^{n(\ell-1)}$, which is the identity on any $\F$-module. 
Its image $Pu$ in $J^{n\ell}(X)$ is also $\pi$-invariant.
Epstein shows in \cite[5.1.3]{Epstein} that $P(u+dv)=Pu+dw$ 
for $v\in I^{n-1}(X)$ and $w\in J^{n\ell-1}(X)$,
so the cohomology class of $Pu$ is independent of the choice of cocycle $u$. 

\begin{defn}\label{def:powermap}
The {\it reduced power map} is defined to be the resulting map on cohomology:
\[  
P: H^n(X,A) \to H_\pi^{n\ell}(X,A^{\otimes\ell}). 
\]
\end{defn}

Now suppose that there is a $\pi$-equivariant map $A^{\oo\ell}\map{m} B$,
and that $\pi$ acts trivially on $B$. (When $A$ is a commutative ring,
multiplication $A^{\oo\ell}\to A$ is a $\pi$-equivariant map.)
We write $m_*$ for the induced map
$H_\pi^*(X,A^{\otimes\ell})\to H_\pi^*(X,B)$.
By Theorem \ref{Kunneth}, $m_*P(u)\in H_\pi^*(X,B)$ has an expansion
$\sum w_k\oo D_k(u)$, where $w_k\in H^k(\pi,\F)$ are the (dual) basis
elements of \cite[V.5.2]{SE}: if $\ell>2$ then $w_0=1$, 
$w_2=\beta w_1$, $w_{2i}=w_2^i$ and $w_{2i+1}=w_1w_2^i$.
If $\ell>2$ and $n\ge2a$, Epstein defines 
\begin{equation}\label{eq:P=D}
P^a:H^n(X,A)\to H^{n+2a(\ell-1)}(X,B), \quad
P^au = (-1)^a\nu_n D_{(n-2a)(\ell-1)}(u), 
\end{equation}
where 
\[
\nu_n=(-1)^r\left(\frac{\ell-1}2\right)!^{-n} \text{ and }
r=\frac{(\ell-1)(n^2+n)}4.
\] 
\noindent (See \cite[7.1]{Epstein}, \cite[VII.6.1]{SE} and \cite{SErr}.)
If $n<2a$ then Epstein defines $P^a=0$.

When $\ell=2$, Epstein defines operations $Sq^i$ by: 
$Sq^i(u)=D_{n-i}(u)$ for $n\ge i$, and $Sq^i(u)=0$ for $n<i$.

\begin{subrem}\label{def:Qa}
Epstein also defines operations 
$Q^a=(-1)^{a+1}\nu_n D_{(n-2a)(\ell-1)-1}(u)$ 
in this setting, and establishes Adem relations for them as well. 

Of course, Epstein's construction mimicks Steenrod's construction 
of $D_k$, $P^a$ and $Q^a$ (see \cite{SE}, VII.3.2 and VII.6.1).
In Steenrod's setting one can lift to integral cochains; with this
assumption, Steenrod proves that $\beta D_{2k}=-D_{2k+1}$ and
hence that $\beta P^a=Q^a$; see \cite[VII.4.6]{SE} and \cite{SErr}.
We will show that the formula $Q^a=\beta P^a$ also holds in our setting
(see Lemma \ref{Bock-etale} and Theorem \ref{Bock-motivic}.)
\end{subrem}
%

\begin{lem}\label{P^a_additive}
For all bounded below chain complexes $A$ and $B$ as above, each function 
$P^a:H^n(X,A)\to H^{n+2a(\ell-1)}(X,B)$ is a group homomorphism.
\end{lem}

\begin{proof}
Corollary 6.7 of \cite{Epstein} applies in this setting.
\end{proof}

\goodbreak
\begin{lem}\label{P^a_natural}
The $P^a$ and $Q^a$ are natural in the map $A^{\oo\ell} \map{m} B$.
\end{lem}

\begin{proof}
Suppose we are given a commutative diagram
\begin{equation*}
\begin{CD}
A_1^{\oo\ell} @>m>> B_1 \\
@VVV @VVV \\
A_2^{\oo\ell} @>m>> B_2.
\end{CD}\end{equation*}
Applying $H_\pi^*(X,-)$ and composing with $P$,
which is natural in $A$ by \cite[5.1.5]{Epstein}, 
Theorem \ref{Kunneth} yields the commutative diagram
\begin{equation*}
\begin{CD}
H^*(X,A_1) @>P>> H_\pi^{*}(X,A_1^{\oo\ell}) @>m>> H_\pi^*(X,B_1)
        @>\cong>> H^*(\pi)\oo  H^*(X,A_1) \\
@VVV @VVV @VVV @VVV \\
H^*(X,A_2) @>P>> H_\pi^{*}(X,A_2^{\oo\ell}) @>m>> H_\pi^*(X,B_2) 
        @>\cong>>  H^*(\pi)\oo H^*(X,A_2).
\end{CD}\end{equation*}
The result now follows from the definition \eqref{eq:P=D} 
of $P^a$ and $Q^a$.
\end{proof}

Recall that the simplicial suspension $SX$ of a simplicial scheme $X$
is again a simplicial scheme. There is a canonical isomorphism
$H_\et^n(X,\muell{i})\map{\cong}H_\et^{n+1}(SX,\muell{i})$.

\begin{prop}\label{S-stable}
The operations $P^a$ are simplicially stable in the sense that they
\index{moved from 3.5 7/30}
commute with simplicial suspension: there are commutative diagrams
for all $X$, $n$ and $i$, with $N=n+2a(\ell-1)$: 
\[\begin{CD}
H_\et^n(X,\muell{i}) @>{P^a}>> H_\et^N(X,\muell{i\ell}) \\
@V{\cong}VV   @V{\cong}VV \\
H_\et^{n+1}(SX,\muell{i}) @>{P^a}>> H_\et^{N+1}(SX,\muell{i\ell}).
\end{CD}\]
\end{prop}

\begin{proof}
The proofs of Lemmas 1.2 and 2.1 of \cite{SE} go through, using
homotopy invariance of \'etale cohomology and excision.
\end{proof}

\section{May's adjoint construction}\label{sec:theta}

A somewhat different approach to constructing cohomology operations was given
by Peter May in \cite{May}. Because we will need May's version of
Kudo's Theorem (in \ref{Kudo} below), we need to know how the two
constructions compare. 

First, we need a chain level version of the Steenrod-Epstein function
$$
m_*P:H^n(X,A)\to H_\pi^{n\ell}(X,A)
$$ 
used in \eqref{eq:P=D} to define $P^a$.  We saw in \ref{def:powermap} 
that the multiplication map $m:A^{\oo\ell}\to A$ lifts to an equivariant map 
$J^*(X)\to C^*\oo I^*(X)$ of their injective resolutions, 
inducing an equivariant map 
\[
\hat{m}: I^{\oo\ell}(X)\to J^*(X) \to C^*\oo I^*(X).
\]
The cohomology function $m_*P$ 
is induced by the chain-level function $u\mapsto\hat{m}(u^{\oo\ell})$.
The expansion $\hat{m}(u^{\oo\ell})=\sum w_k\oo D_k(u)$ in 
$C^*\oo I^*(X)$ effectively defines functions 
$D_k: I^n(X)\to I^{n\ell-k}(X)$.

Consider the isomorphism $\phi:C^*\oo I^*(X)\to\Hom(C_*,I^*(X))$, 
defined by 
\[ 
\phi(f\oo x)(c)=(-1)^{|x|\,|c|} f(c)x.
\] As a special case, 
$\phi(w_j\oo x)(e_k)=(-1)^{k|x|}\delta_{jk}\,x.$
The composition $\phi\,\hat{m}$ sends $I^{\oo\ell}(X)$ to $\Hom(C_*,I^*(X))$.
It is the (signed) adjoint of the map $\phi\,\hat{m}$, 
\begin{equation}\label{def:theta}
\theta:C_*\oo I^{\oo\ell}(X)\to I^*(X),
\end{equation}
which forms the basis for May's approach; see \cite[2.1]{May}. 
May defines the function $D^M_k:I^{n}(X)\to I^{n\ell-k}(X)$ by the formula
$$
D^M_k(u)=\theta(e_k\oo u^{\oo\ell}). 
$$
%
%
May defines $P_M^a$ and $Q_M^a$ by 
\[
P_M^a(u)=(-1)^a\nu_n D^M_{(n-2a)(\ell-1)}(u)
\hskip5pt\text{and}\hskip5pt
Q_M^a(u)=(-1)^a\nu_n D^M_{(n-2a)(\ell-1)-1}(u)
\]
(see \cite[pp.\,162,\,182]{May}; his $\nu(-n)$ is our $\nu_n$). 
The sign differences in the formulas for $P^a$ and $P_M^a$ 
(and for $Q^a$ and $Q_M^a$) are explained by the following calculation.

\begin{prop}
For $u$ in $H^n(X,A)$, $D^M_k = (-1)^k D_k$.
\end{prop}

\begin{proof}
Consider the isomorphism $\phi:C^*\oo I^*(X)\to\Hom(C_*,I^*(X))$, 
defined above. 
The adjoint $\theta$ of $\phi\,\hat{m}$ is the composite
\[
C_*\oo I(X)^{\oo\ell} \cong I(X)^{\oo\ell}\oo C_* 
\map{\phi\,\hat{m}\oo1}
\Hom(C_*,I(X))\oo C_* \map{\eta} I(X),
\]
where the first map is the signed symmetry isomorphism 
and $\eta$ is evaluation. We now compute that 
\begin{align*}
D^M_k(u) = \theta(e_k\oo u^{\oo\ell}) 
        &= (-1)^{kn\ell}\eta\left(\phi[\hat{m}(u^{\oo\ell})]\oo e_k\right)\\
&= (-1)^{kn\ell}\eta\left(\phi\left[\sum w_j\oo D_j(u)\right]\oo e_k\right) \\
        &= (-1)^{kn\ell}\sum_j\phi\left[w_j\oo D_j(u)\right](e_k) \\
        &= (-1)^{kn\ell} (-1)^{k(n\ell-k)} D_k(u) \\
        &=  (-1)^k D_k(u).   
\qedhere
\end{align*}
\end{proof}

\begin{cor}\label{PQagree}
May's operations $P_M^a$ and $Q_M^a$ coincide with the $P^a$ and $Q^a$
of \eqref{eq:P=D} and \ref{def:Qa}.
\end{cor}

\begin{lem}\label{dP-Pd}
Set $m=(\ell-1)/2$. Then for each $u\in I^n(X)$: \\
 (i) $dP^a(u)= P^a(du)$ and $dQ^a(u)=- Q^a(du)$, and \\
(ii) if $u$ is a cocycle representing $x\in H^n(X,A)$ then 
$P^a(u)$ and $Q^a(u)$ are cocycles representing $P^a(x)$ and $Q^a(x)$,
respectively.
\end{lem}

\begin{proof}
In Theorem 3.1 of \cite{May}, May shows that
(i) $dP_M^a(u)=P_M^a(du)$ and $dQ_M^a(u)=-Q_M^a(du)$, and 
(ii) if $u$ is a cocycle representing $x\in H^n(X,A)$ then 
$P_M^a(u)$ and $Q_M^a(u)$ are cocycles representing $P_M^a(x)$ and $Q_M^a(x)$.
The result is immediate from Corollary \ref{PQagree}.
\end{proof}

\begin{rem}
In \cite{May1}, May gave a different approach to power operations in sheaf
cohomology. If $A$ is any sheaf of $\F$-algebras, May shows (in 3.12) 
that the sections over $X$ of the Godement resolution $F^\mathdot A$
of $A$ yield an algebra $C^\mathdot = \Hom_\Delta(\Lambda,F^\mathdot A)$ 
over the Eilenberg-Zilber operad $\mathcal I$ on cochains of $\F$-modules, 
which therefore inherits the structure of an $E_\infty$-algebra. 
Since the cohomology of $C^\mathdot$ is the sheaf cohomology $H^*(X,A)$,
the technique in \cite{May} produces cohomology operations.
\end{rem}

\newpage
\section{The \'etale Steenrod algebra}
\label{sec:Aetale}

In this section we determine the algebra of all \'etale cohomology
operations $H_\et^n(-,\muell{i})\to H_\et^*(-,\muell{j})$ over a 
field $k$ containing $1/\ell$.

Recall from SGA 4 (V.2.1.2 in \cite{Ver}) that if $M$ is a (simplicial)
\'etale sheaf of $\F$-modules then the sheaf cohomology groups
$H_\et^*(X,M)$ are isomorphic to the (hyper) $\Ext$-groups 
$\Ext^*(\F[X],M)$ in the category of \'etale sheaves of $\F$-modules.
(Here we regard $M$ as a cochain complex using Dold-Kan.)
If $K$ is a second simplicial \'etale sheaf of $\F$-modules, 
one writes $H_\et^*(K,M)$ for $\Ext^*(K,M)$.

It is well known that cohomology operations $H_\et^n(-,L)\to H_\et^*(-,M)$
are in 1--1 correspondence with elements of $H_\et^*(K,M)$,
where $K$ denotes the standard simplicial Eilenberg-Mac\,Lane sheaf
$K(L,n)$ associated to $L$. 

We first discuss the case of constant coefficients ($M=\F$),
which is known, and due to Breen \cite[4.3--4]{Breen} 
and Jardine \cite{Jardine}.  
The graded ring of all unstable \'etale cohomology
operations from $H_\et^n(-,\F)$ to $H_\et^*(-,\F)$ is isomorphic
to the cohomology ring $H_\et^{*}(K_n,\F)$, where $K_n=K(\F,n)$ 
is the constant simplicial sheaf classifying elements of $H_\et^n(-,\F)$.
By Theorem \ref{etale-exist}, there is a ring homomorphism from
the classical unstable Steenrod algebra $H^*_\top(K_n)$ of
Definition \ref{def:PI} to $H_\et^{*}(K_n,\F)$.

There is also a ring homomorphism from $H_\et^*(k,\F)$ to $H_\et^{*}(K_n,\F)$. 
These two ring homomorphisms do not commute, because the Bockstein 
and other operations $P^I$ may be nontrivial on $H_\et^*(k,\F)$. 
If $c\in H_\et^*(k,\F)$ then $cP^I$ and $P^Ic$ send $x$ to 
$c\cdot P^I(x)$ and to $P^I(cx)$, respectively.
Nevertheless, we can define a twisted multiplication on
$H_\et^*(k,\F)\oo_{\F}H^*_\top(K_n)$ using the Cartan relations
$P^a\circ\alpha=\sum P^i(\alpha)P^j$, $\alpha\in H_\et^*(k,\F)$.  
We shall refer to this non-commutative algebra as the 
{\it twisted tensor algebra}.
It is free as a left $H_\et^*(k,\F)$-module, and a basis
is given by the monomials in the Steenrod operations $P^I$ and $\beta P^I$
where $I$ has excess $<n$, exactly as in the topological case.
We summarize this:


\begin{thm}\label{BJ}
The ring of \'etale cohomology operations on $H_\et^n(-,\F)$ is
the twisted tensor product $H_\et^*(k,\F)\otimes H^*_\top(K_n)$: 
every operation is a polynomial
in the operations $P^I$ with coefficients in $H_\et^*(k,\F)$.
\end{thm}
%

\begin{subex}
When $k=\mathbb R$ and $\ell=2$, the ring of \'etale cohomology
operations over $\mathbb R$ is the graded polynomial ring
$H^*_\top(K_n)[\sigma]$ with generator $\sigma$ in degree 1 and all
$Sq^I(\iota_n)$ with $I$ admissible and $e(I)<n$. This is because
$H_\et^*({\mathbb R},{\mathbb F}_2)={\mathbb F}_2[\sigma]$.
\end{subex}

\begin{subrem}\label{et-top}
When $k=\mathbb C$, the action of the $P^I$ is compatible with 
the canonical comparison isomorphism 
$H_\et^*(X,\F)\cong H_\top^*(X(\mathbb C),\F)$.
This is clear from
\edit{added 12/28}
the constructions in \cite{Epstein} and \cite{Jardine}.
\end{subrem}

When $k$ contains a primitive $\ell^{th}$ root of unity,
the sheaves $\muell{i}$ are all isomorphic. Thus the ring of operations
$H_\et^n(-,\muell{i})\to H_\et^*(-,\muell{j})$ is isomorphic to
$H_\et^*(k,\F)\otimes H^*_\top(K_n)$ as a left $H_\et^*(k,\F)$-module.
Since this is always the case when $\ell=2$, we shall
restrict to the case of an odd prime $\ell$. 


Now fix $i$ and consider cohomology operations 
$H_\et^n(-,\muell{i})\to H_\et^*(-,\muell{j})$.
As observed above, they are in 1--1 correspondence with 
elements of $H_\et^*(K,\muell{j})$, where
$K$ denotes the simplicial Eilenberg-Mac\,Lane scheme
$K(\muell{i},n)$. 
For example, the identity operation on $H_\et^n(-,\muell{i})$
corresponds to $\iota_n\in H^n(K,\muell{i})$, and the \'etale Bockstein 
$\beta:H_\et^n(X,\muell{i})\to H_\et^{n+1}(X,\muell{i})$ corresponds
to $\beta(\iota_n)\in H^{n+1}(K,\muell{i})$.



Fix a field $k$ with $1/\ell\in k$, and 
let $G$ be the Galois group of the extension $k(\zeta)/k$, 
where $\zeta$ denotes a primitive $\ell^{\text th}$ root of unity. 
Then $G$ is cyclic of order $d=[k(\zeta):k]$, $d\mid\ell-1$,
and $\muell{d}\cong\F$. Since $i\equiv j\pmod d$ implies that
$\muell{i} \cong \muell{j}$, we are led to
consider the $\Z/d$-graded \'etale sheaf of rings
$\cO=\bigoplus_{i=0}^{d-1} \muell{i}$.
Thus our problem is to determine the ring
$H_\et^*(K,\cO)=\oplus_{j=0}^{d-1} H_\et^*(K,\muell{j})$.

\goodbreak
Since $\ell$ does not divide $|G|$, Maschke's Theorem 
gives an identification of the \'etale sheaf $\F[G]$ with the 
direct sum of the sheaves of irreducible $\F[G]$-modules $\muell{i}$, 
\ie with $\cO$.
For any $X$, Shapiro's Lemma provides an isomorphism
\begin{equation}\label{eq:Shapiro}
H_\et^*(X,\cO)\cong H_\et^*(X(\zeta),\F),
\end{equation}
where $X(\zeta)$ denotes $X\times_k \Spec(k(\zeta))$. 
In fact, $\cO=Ri_*\F$, where $i:\Spec(k(\zeta))_\et\to\Spec(k)_\et$.

Any $\F[G]$-module $M$ is the sum of its isotypical summands, the 
isotypical summand for $\muell{i}$ being $\Hom_G(\muell{i},M)$.
In particular, the action of $G$ on $X(\zeta)$ decomposes 
$H_\et^*(X(\zeta),\F)$ into its isotypical pieces.
Because the $\muell{i}$ are the isotypical summands of $\cO=Ri_*\F$,
the summand $H_\et^*(X,\muell{i})$ in \eqref{eq:Shapiro} is the 
isotypical summand of $H_\et^*(X(\zeta),\F)$ for $\muell{i}$.

This is the context in which Epstein defines operations $P^a$ and $Q^a$
via \eqref{eq:P=D}.
Because the Frobenius is the identity on $\cO$, $P^0$ is the identity 
operation. We can now show that Epstein's operation $Q^0$ is the 
\'etale Bockstein $\beta$, and his $Q^a$ is $\beta P^a$.

\begin{lem}\label{Bock-etale} 
In \'etale cohomology, $Q^0=\beta$ and $Q^a=\beta P^a$ for $a>0$.
\end{lem}

\begin{proof}
We first consider the case when $\zeta\in k$, so that $\muell{i}\cong\F$
for all $i$.
Jardine's argument in \cite[pp.\,108--114]{Jardine} that Epstein's
$Sq^1$ is the Bockstein when $\ell=2$ applies when $\ell>2$ as well,
and proves that Epstein's $Q^0$ is the Bockstein operation.
The identity $Q^0=\beta$ in the general case 
follows from this and the isomorphism \eqref{eq:Shapiro}:
\[\begin{CD}
H_\et^n(X,\cO) @>{Q^0-\beta}>> H_\et^{n+1}(X,\cO) \\
@V{\cong}VV @VV{\cong}V \\
H_\et^n(X(\zeta),\F) @>{Q^0-\beta}>> H_\et^{n+1}(X(\zeta),\F).
\end{CD}\]
Finally, we invoke the Adem relation $Q^0P^b=Q^bP^0$ \cite[9.8(4)]{Epstein}
to get $\beta P^b=Q^b$.
\end{proof} 

Using the Bockstein and Epstein's operations $P^a$, we have operations $P^I$
defined on $H_\et^n(-,\muell{i})$
for every admissible sequence $I$ in the sense of Definition \ref{def:PI}.

In order to classify all operations on $H_\et^n$, we first consider 
the case $n=1$.  In topology, the ring of operations on $H^1(-,\F)$ is 
$H_\top^*(K_1)\cong\F[u,v]/(u^2)$, where $u=P^0$ is in degree $1$, 
corresponding to the identity operation, and $v$ is in degree $2$,
corresponding to the Bockstein operation. By Theorem \ref{etale-exist},
there is a canonical map from $\F[u,v]/(u^2)$ to \'etale cohomology operations 
from $H_\et^1(-,\muell{i})$ to $H_\et^*(-,\muell{*})$, 
sending $u$ to the identity and $v$ to
the Bockstein $\beta:H_\et^1(-,\muell{i})\to H_\et^2(-,\muell{i})$.

For any $i$, the basechange $\muell{i}(\zeta)$ of the algebraic group 
$\muell{i}$ is isomorphic to $\F(\zeta)$, the constant sheaf $\F$ 
on the big \'etale site of $k(\zeta)$.
The induced isomorphism $(B\muell{i})(\zeta) \cong (B\F)(\zeta)$ 
induces an isomorphism of cohomology groups, which immediately 
yields the following calculation.


\begin{prop}\label{Het1-ops}
The graded algebra of cohomology operations from $H_\et^1(-,\muell{i})$
to $H_\et^*(-,\cO)=\oplus_{j=0}^{d-1} H_\et^*(-,\muell{j})$ 
is isomorphic to the  $H_\et^*(k(\zeta),\F)$-module
\[
H^{*}(B\muell{i},\cO) \cong H^*(B\muell{i}(\zeta),\F) \cong 
H^{*}(k(\zeta),\F)\oo\F[u,v]/(u^2),\ \beta(u)=v.
\]
\noindent
Every operation on $H_\et^1(-,\muell{i})$ is uniquely 
a sum of operations $\phi(x)=cx^\varepsilon\beta(x)^m$, where 
$c\in H_\et^*(k,\muell{j})$ and $\varepsilon\in\{0,1\}$.
\end{prop}

Proposition \ref{Het1-ops} is the case $n=1$ of the following result.

\begin{thm}\label{thm:etaleops}
For each $i$ and $n\ge1$, the ring of all \'etale cohomology operations 
from $H_\et^n(-,\muell{i})$ to $H_\et^*(k,\cO)$
is the free left $H_\et^*(k(\zeta),\F)$-module
$H_\et^*(k(\zeta),\F)\otimes \tilde{H}^*_\top(K_n)$.

The operations $P^I$ send $H_\et^n(-,\muell{i})$ to $H_\et^*(-,\muell{i})$.
Thus the operations from $H_\et^n(-,\muell{i})$ to $H_\et^*(-,\muell{i+j})$
are isomorphic to $H_\et^*(-,\muell{j})\otimes\tilde{H}^*_\top(K_n)$.
\end{thm}

\begin{proof}
We first show that the basechange $K(\muell{i},n)\times_k\Spec(k(\zeta))$
is the space $K(\muell{i},n)$ over $k(\zeta)$.
This is clear for $n=0$, and follows inductively from the construction
of $K(A,n+1)$ via the bar construction on $K(A,n)$, together with the 
observation that $(X\times_k Y)\times_k\Spec(k(\zeta))$
is $X(\zeta)\times_{k(\zeta)}Y(\zeta)$.

By \eqref{eq:Shapiro}, the cohomology of $K(\muell{i},n)$ with 
coefficients in $\cO$ is the same as the cohomology of 
$K(\muell{i},n)\times_k\Spec(k(\zeta))$ with coefficients in $\F$. 
The Breen-Jardine result, Theorem \ref{BJ},
shows that this is $H_\et^*(k(\zeta),\F)\otimes H_\top^*(K_n)$.
\end{proof}

\bigskip
\section{Motivic Steenrod operations}\label{sec:Amotivic}

In this section we construct operations $P^a$ on the motivic cohomology
groups $H^{n,i}(X)=H^{n,i}(X,\F)$, $n\ge2a$,
compatible with the operations $P^a$ in \'etale cohomology 
in the sense that there are commutative diagrams
\begin{equation}\label{eq:compatible}
\begin{CD}
H^{n,i}(X,\F) @>P^a>> H^{n+2a(\ell-1),i\ell}(X,\F) \\
@VVV @VVV \\
H_\et^n(X,\muell{i}) @>P^a>>
    H_\et^{n+2a(\ell-1)}(X,\muell{i\ell}).
\end{CD}\end{equation}

Let $\alpha_*$ denote the direct image functor from the \'etale site to the
Nisnevich site. If $\cF$ is any \'etale sheaf then we may regard 
$A=R\alpha_*\cF$ as a complex of Nisnevich sheaves such that
$H_\nis^*(X,R\alpha_*\cF)\cong H_\et^*(X,\cF)$.

The following theorem, due to Voevodsky and Rost, is sometimes known
as the Beilinson Conjecture; it is equivalent to the Norm Residue Theorem;
see \cite{SV00}, \cite{V-MC/l}, \cite{W09}), \cite{HW}).
Let $\tau^{\le i}A$ denote the good truncation of $A$ in cohomological
degrees at most $i$; $H^n(\tau^{\le i}A)$ 
is $H^n(A)$ for $i\le n$, and zero for $n>i$.
(cf.\ \cite[1.2.7]{WHomo}).

\begin{NormResidueThm}\label{NRT}
For any $X$ smooth over a field of characteristic $\ne\ell$,
the map $\F(i)\to \tau^{\le i}R\alpha_*\muell{i}$
is a quasi-isomorphism, and hence
\[
H^{n,i}(X,\F)\cong H_\nis^n(X,\tau^{\le i}R\alpha_*\muell{i}).
\]
In particular, if $n\le i$ then $H^{n,i}(X,\F)\cong H_\et^n(X,\muell{i})$.
\end{NormResidueThm}

\smallskip
Explicitly, if $\cF$ is an \'etale sheaf and
$\cF\to I^*$ is an injective resolution, 
then $\tau^{\le i}R\alpha_*\cF$ is represented by 
$\tau^{\le i}\alpha_*(I^*)$. It is quasi-isomorphic to a chain complex 
$I_\nis^*$ of injective Nisnevich sheaves on $X$ with $I_\nis^n=\alpha_*I^i$ 
for $n\le i$, because each $\alpha_*I^n$ is an injective Nisnevich sheaf
and $Z^n(\alpha_*I)$ injects into $\alpha_*I^n$. Taking $\cF=\muell{i}$, 
the theorem states that $H^{n,i}(X,\F)=H_\nis^n(X,\F(i))$ is the 
$n^{th}$ cohomology of $I_\nis^*(X)$.

\smallskip
Now the product on $\muell{i}$ yields a product
$R\alpha_*\muell{i}\otimes R\alpha_*\muell{j}\to R\alpha_*\muell{i+j}\!,$ 
unique up to chain homotopy, by the comparison theorem. 
The induced pairing $\tau^{\le i}R\alpha_*\muell{i}\oo
\tau^{\le j}R\alpha_*\muell{j}\to \tau^{\le i+j}R\alpha_*\muell{i+j}$
induces the product in motivic cohomology, by \cite[7.1]{SV00}.
We may choose a model for the $\ell$-fold product map
$(R\alpha_*\muell{i})^{\oo\ell}\to R\alpha_*\muell{i\ell}$
which is equivariant for the permutation action of the cyclic group $\pi$,
by choosing a $\pi$-equivariant replacement 
$R\alpha_*\muell{i\ell}\map{\sim}J$ and factoring through an
equivariant map $(R\alpha_*\muell{i})^{\oo\ell}\to J$.
The truncation of this map is also equivariant:
\[
(\tau^{\le i}R\alpha_*\muell{i})^{\oo\ell} \to
\tau^{\le i\ell}(R\alpha_*\muell{i})^{\oo\ell} 
\to \tau^{\le i\ell}R\alpha_*\muell{i\ell}.
\]

Setting $A=\oplus_{i=0}^\infty \tau^{\le i}R\alpha_*\muell{i}$,
we have a graded map $A^{\otimes\ell}\map{m} A$, representing the 
product in motivic cohomology.
Composing the power map $P$ of Definition \ref{def:powermap}, the map $m$ 
and the isomorphism of Theorem \ref{Kunneth}, 
we have a graded reduced power map on $H^{n,i}(X,A)=\oplus H^{n,i}(X,\F)$:
\[
H^{n}(X,A)\map{P} H_\pi^{n\ell}(X,A^{\otimes\ell})\ \map{m_*}\  
H_\pi^{n\ell}(X,A) \cong  \oplus H^{n\ell-k}(X,A)\otimes  H^k(\pi,\F).
\]

\begin{defn}[$\mathbf P^a$]\label{motivic-exist}
The function $P^a: H^{n,i}(X) \to H^{n+2a(\ell-1),i\ell}(X)$
is defined as follows. Given $u\in H^{n,i}(X)$, $m_*P(u)$ has an expansion
$\sum D_k(u)\otimes w_k$, where $w_k\in H^k(\pi,\F)$ are as before.
If $\ell\ne2$ and $n\ge2a$, we define $P^au$ to be $(-1)^a\nu_n$ times
$D_{(n-2a)(\ell-1)}u$, where the constant $\nu_n$ is given by 
the formula in \eqref{eq:P=D}. 
If $n<2a$ we define $P^a=0$.
By Lemma \ref{P^a_additive}, each $P^a$ is in fact a homomorphism.

We call the $P^a$ {\it motivic cohomology operations;}
they 
are natural in $X$, by the argument of
Remark \ref{natural-etale} applied to the power map on $H^{n,i}(X,A)$. 

If $\ell=2$ we define $Sq^a:H^{n,i}(X)\to H^{n+a,2i}(X)$ to be 
$D_{n-a}$ for $n\ge a$, and $Sq^a=0$ for $n<a$. This follows
Steenrod and Epstein. Thus $Sq^{2a}=P^a$. 
We will show in Theorem \ref{Bock-motivic} below that 
$Sq^{2a+1}=\beta Sq^{2a}$.
\end{defn}\goodbreak

\begin{subrem}
These motivic cohomology operations are almost surely the operations
defined by Kriz and May in \cite[I.7.2]{KM}, and by Joshua in 
\cite[\S8]{Joshua}; cf.\ \cite{BJ}.
\end{subrem}

\begin{lem}\label{lem:mot-et}
The motivic cohomology operations $P^a$ are compatible with the
\'etale cohomology operations $P^a$ 
in the sense that the diagram \eqref{eq:compatible} commutes.
\end{lem}

\begin{proof}
By construction, the following diagram commutes:
\begin{equation*}
\begin{CD}
(\tau^{\le i}R\alpha_*\muell{i})^{\oo\ell} @>m>> 
 \tau^{\le i\ell}R\alpha_*\muell{i\ell} \\ @VVV @VVV \\
(R\alpha_*\muell{i})^{\oo\ell} @>m>> R\alpha_*\muell{i\ell}.
\end{CD}\end{equation*}
The commutativity of \eqref{eq:compatible} now follows 
from Lemma \ref{P^a_natural}.
\end{proof}

\goodbreak
We now show that these operations enjoy familiar properties.
\smallskip

\begin{prop}\label{l-power}
If $u\in H^{2n,i}(X)$ then $P^n(u)=u^\ell$. 
\end{prop}

\begin{proof}
By \cite[5.2.1]{Epstein}, $j_*:H^*_\pi(X,R)\to H^*(X,R)$ sends
$Pu$ to $u\times\cdots\times u$. Since $w_0=1$, the proof 
in \cite[6.3, 7.3]{Epstein} applies.
\end{proof}

\begin{lem}\label{E4.4.3}
Let $\pi$ and $\rho$ be finite subgroups of $\Sigma_\ell$.
Then their action on $R^{\otimes\ell}$ induces a commutative diagram
(whose vertical maps come from Theorem \ref{Kunneth}):
\[\begin{CD}
H^*(\pi)\oo H^*(\rho)\oo H^*(X,R)\oo H^*(X,R) @>>>
H^*(\pi\times\rho)\oo H^*(X,R\oo R) \\
@V{\cong}VV @VV{\cong}V \\
H^*_\pi(X,R)\oo H^*_\rho(X,R) @>>> 
H^*_{\pi\times\rho}(X,R\oo R).
\end{CD}\]
\end{lem}

\begin{proof}
The proof of Lemma 4.4.3 of \cite{Epstein}, which concerns the underlying
chain complexes rather than the cohomology groups, goes through.
\end{proof}

\begin{thm}[Cartan Formula]\label{Cartan}
Let $u\in H^{n,i}(X)$ and $v\in H^{m,j}(Y)$. 
Then in $H^{*,(i+j)\ell}(X\times Y)$ we have:
\[
P^a(u \cup v) = \sum_{s+t=a} P^s(u) \cup P^t(v),\quad \ell>2,
\]
and $Sq^a(u \cup v) = \sum_{s+t=a} Sq^s(u) \cup Sq^t(v)$ when $\ell=2$.
\end{thm}

\begin{proof}
Epstein's proof in \cite[7.2]{Epstein} carries over. In more detail,
replacing \cite[4.4.3]{Epstein} by our Lemma \ref{E4.4.3} in the proof
of Lemma 7.2.2 in \cite{Epstein}, we obtain the formulas
$$
D_{a}(u \cup v)=\sum_{s+t=a}\pm D_{s}u\cup D_{t}v, 
$$
where $s$ and $t$ cannot both be odd.
By \cite[6.4]{Epstein}, if $n$ is even (resp., odd) then 
$D_s(u)=0$ unless $s$ is $m(\ell-1)$ or $m(\ell-1)-1$ 
for some even (resp., odd) integer $m\ge0$.
The Cartan Formula follows by inspection of the signs involved.
\end{proof}

\begin{subrem}
Epstein also establishes a Cartan formula for the operations $Q^a$,
including the formula $Q^0(uv)=(Q^0u)(P^0v)+(-1)^{|u|}(P^0u)(Q^0v)$.
We omit these formulas, as they follow from Theorem \ref{Cartan}
and the formulas for $Q^a$ in Theorem \ref{Bock-motivic} below.
\end{subrem}

We now turn to the Adem relations. Recall that by convention
$\choose{n}{k}$ is zero if $k<0$. Thus the sums below run over $t\le a/\ell$.

\begin{thm}[Adem Relations]\label{Adem}
If $\ell>2$ and $a<b\ell$ then
\begin{align*}
P^aP^b =& \sum_{s+t=b} (-1)^{a+t} \choose{(\ell-1)s-1}{a-t\ell}\ P^{a+s}P^t;
\\      P^a \beta P^b = &
\sum_{s+t=b} (-1)^{a+t}  \choose{(\ell-1)s}{a-t\ell}\ \beta P^{a+s}P^t +
(-1)^{a+t} \choose{(\ell-1)s-1}{a-t\ell-1}\ \beta P^{a+s}\beta P^t.
\end{align*}
\end{thm}

\begin{proof}[Proof (Epstein)]
We refer to section~9 of \cite{Epstein}, whose running assumption is
that the operations $P^a$ (and $Q^a$) are zero on $H^{*,*}$ for $a<0$.
This assumption holds by \cite[8.3.4]{Epstein}, using the adjunction
for sheaves in \cite[11.1]{Epstein}. In particular, 
$P^a$ and $Q^a$ vanish on $H_\pi^{*,*}$ by \cite[9.1]{Epstein}.
The proof of \cite[9.3]{Epstein} only requires an equivariant map 
$A^{\oo\ell}\to A$, so 9.3 and its Corollary 9.4 of {\em loc.\,cit.\,}\ 
remain valid in the motivic setting. Since 9.2, 9.5 and 9.6 of
\cite{Epstein} are formally true, we can now quote
the proof of \cite[9.8]{Epstein}: the proof of the Adem relations
on pp.\,119--122 of \cite{SE}, as amended by the Errata,  
carry over to this setting.
\end{proof}

\begin{subrem}
When $\ell=2$, Epstein points out in \cite[9.7]{Epstein} that, 
given the modifications in the proof of \ref{Adem} above, 
the usual Adem relations hold by the proof on p.\,119 of \cite{SE}: 
if $a<2b$ then 
\[
Sq^a Sq^b = \sum_{s+t=b} \choose{s-1}{a-2t}\ Sq^{a+s}Sq^t.
\]
\end{subrem}

\begin{subrem}
The usual Adem relations for $Q^aQ^b$ and $Q^aP^b$, stated in 
\cite[9.8]{Epstein}, are also valid in motivic cohomology.  These 
relations follow immediately from Theorem \ref{Adem},
given the formulas for $Q^a$ in Theorem \ref{Bock-motivic} below.
\end{subrem}

%

\begin{bistable}\label{P_V ops}
In \cite{RPO}, Voevodsky defines (bistable) cohomology operations $P_V^a$ on
$H^{n,i}(-,\F)$ of bidegree $(2a(\ell-1)$, $a(\ell-1))$. These satisfy:
$P_V^0x=x$ for all $x$; $P_V^ax=x^\ell$ for $x\in H^{2a,a}(X,\F)$;
$P_V^a=0$ on $H^{n,i}$ if $i\le a$ and $n<i+a$; the usual Adem relations
hold when $\ell>2$. The cohomological degrees of $P^a$ and $P_V^a$ are
the same, namely $2a(\ell-1)$, but the weights differ if $a\ne i$: 
if $a<i$ then $P_V^a$ has lower weight, but 
if $a>i$ then $P^a$ has lower weight.

When $\ell=2$, Voevodsky's operations $Sq_V^{2i}$ have bidegree$(2i,i)$
and $Sq_V^{2i+1}=\beta Sq^{2i}$.  They satisfy a modified Cartan formula
\cite[9.7]{RPO} which differs from our Cartan formula 
(Theorem \ref{Cartan}) by the presence of a factor of $[\zeta]$ in some terms.
\end{bistable}

\begin{subrem}\label{PV-Pet}
Brosnan and Joshua have observed in \cite[2.1]{BJ} 
and \cite[1.1(iii)]{BJ1} that the motivic-to-\'etale map 
sends $P_V^a$ to $P^a$ and $Sq_V^a$ to $Sq^a$.
The key is to observe that Voevodsky's total power operation 
\cite[5.3]{RPO} is compatible with Epstein's reduced power map 
(Definition \ref{def:powermap} above).
\end{subrem}

Cohomology operations on $H^{n,0}$ are easy to describe because
of the following characterization.

\begin{lem}\label{twist0}
Let $A$ be any abelian group.
If $X_\mathdot$ is a smooth simplicial scheme, the motivic cohomology ring
$H^{*,0}(X_\mathdot,A)$ is isomorphic to the topological cohomology
$H_\top^*(\pi_0X_\mathdot,A)$ of the simplicial set $\pi_0(X_\mathdot)$.
\end{lem}

\begin{proof}
For smooth connected $X$ we have $H^{n,0}(X,A)=H^n_\nis(X,A)$ for $n>0$
and $H^{0,0}(X,A)=A$, almost by definition; see \cite[3.4]{MVW}.
Hence the spectral sequence $E_1^{p,q}=H^q(X_p,A)\Rightarrow H^{p+q,0}(X)$
\edit{moved here 12/28}
degenerates to the cohomology of the chain complex 
$\Hom(\pi_0(X_\mathdot),A)$, which is $H_\top^*(\pi_0X_\mathdot,A)$.
For a simplicial set $K$ such as $\pi_0X_\mathdot$, 
the construction of the product in motivic cohomology \cite[3.11]{MVW}
shows that $H_\top^*(K)\cong H^{*,0}(K)$ is an isomorphism of rings.
\end{proof}

\begin{cor}
The ring of motivic cohomology operations on $H^{n,0}(-,\F)$ is 
isomorphic to $H^{*,*}(k,\F)\oo H_\top^*(K_n)$.
\end{cor}

If $K$ is a simplicial set, the isomorphism
$H^{*,0}(K,\F)\cong H_\top^*(K,\F)$ is compatible with the action
of the $P^I$. This is clear from Lemma \ref{lem:mot-et} 
and Remark \ref{et-top}.

\begin{subex}\label{suspension elt}
Let $\Delta^1$ denote the simplicial 1-simplex and 
$s\in H^{1,0}(\Delta^1,\partial\Delta^1)$ the generator.
By the above comparison with topology, $P^0(s)=s$.
By definition, $P^a(s)=0$ for $a>0$.
\end{subex}

Recall that the simplicial suspension $SX$ of a pointed simplicial 
scheme $X$ is again a simplicial scheme. Multiplication by the element $s$
of Example \ref{suspension elt} induces a canonical isomorphism
$H_\et^n(X,\muell{i})\map{\cong}H_\et^{n+1}(SX,\muell{i})$.
(Compare to Lemmas 1.2 and 2.1 of \cite{SE}.)

\begin{prop}\label{S1-stable}
The motivic operations $P^a$ are simplicially stable in the sense that they
\index{added 12/28}
commute with simplicial suspension: there are commutative diagrams
for all $X$, $n$ and $i$, with $N=n+2a(\ell-1)$: 
\[\begin{CD}
H^{n,i}(X) @>{P^a}>> H_\et^{N,i\ell}(X) \\
@V{\cong}VV   @V{\cong}VV \\
H^{n+1,i}(SX) @>{P^a}>> H^{N+1,i\ell}(SX).
\end{CD}\]
\end{prop}

\begin{proof}
By the Cartan formula \ref{Cartan}, $P^a(sx)=P^0(s)P^a(x)=s\cdot P^a(x)$.
\end{proof}

\begin{ex}
Consider the classifying space $K_{\F}=K(\F(i),n)$ for $H^{n,i}(-,\F)$.
If $n\ge i$, we see from \cite[3.27]{V11} that the summands of smallest 
weight or degree in $\Ftr(K_{\F})$ are $\F(i)[n]$ and $\F(i)[n+1]$.
It follows that: $H^{a,*}(K_{\F})=H^{a,*}$ for $a<n$;
$H^{n,*}(K_{\F})\cong H^{0,*}\oplus H^{n,*}$
on the tautological class $\iota$ in $H^{n,i}(K_{\F})$ and
constant maps to elements of $H^{n,*}(k,\F)$; and
$H^{0,0}\oplus H^{1,*}\oplus H^{n+1,*}\map{\cong}H^{n+1,*}(K_{\F})$ 
by the map $(c,c',c'')\mapsto c\beta(\iota)+c'\iota+c''$. That is, 
there are no cohomology operations $H^{n,i}\to H^{p,q}$ with $p\le n$ 
except for constant operations and $\F$-linear terms when $p=n$
(multiples of the identity plus a constant),
and no operations $H^{n,i}\to H^{n+1,q}$ other than the Bockstein,
multiplication by elements of $H^{1,*}(k)$, and constants.

If $n<i$, this is no longer the case. 
In Example \ref{ex:b=1} below,
we show that there is a weight-reducing operation 
$H^{1,2}\to H^{2,1}$ for all $k$, and a weight-preserving operation
$H^{1,2}\to H^{3,2}$ for most $k$.  For another example, suppose that
$\zeta\in k$ and $n\le i$. Then cupping with $[\zeta]\in H^{0,1}(k)$ is an 
isomorphism by Theorem \ref{NRT}; its inverse (defined when $n<i$) is 
an operation $H^{n,i}\to H^{n,i-1}$.
\end{ex}

\newpage 
\section{The operations $P^0$ and $Q^0$} 

Sometimes we can deduce motivic operations from \'etale operations.
For example, if $n\le i$ (and hence $n\le i\ell$) then the diagram
\eqref{eq:compatible} allows us to identify the motivic operation 
$P^0:H^{n,i}(X)\to H^{n,i\ell}(X)$ with the \'etale operation 
$P^0:H_\et^n(X,\muell{i})\cong H_\et^n(X,\muell{i\ell})$, and thus 
conclude that $P^0$ is an isomorphism in this range. 
The same reasoning, using the Norm Residue Theorem \ref{NRT}, 
shows that if $n\le i$ and $n+2a(\ell-1)\le i\ell$, 
the motivic and \'etale operations $P^a$ agree on 
$H^{n,i}(X)\cong H_\et^n(X,\muell{i})$, and also agree with 
$b^{(i-a)(\ell-1)/d}P_V^a$, where $b\in H^{0,d}(k)$ is defined as follows.

Fix a primitive $\ell^{th}$ root of unity, $\zeta$, in an extension
field of $k$; this choice determines a canonical generator $[\zeta]$
of $H^0(k(\zeta),\mu_\ell)$. If $[k(\zeta):k]=d$ then 
$H^{0,d}(k)\cong H_\et^0(k,\muell{d})$, and the element 
$[\zeta]^d=[\zeta^{\oo d}]$  descends to a canonical ``periodicity''
element $b$ in $H^{0,d}(k)$. (If $d=1$ then $b=[\zeta]$.)
Note that multiplication by $b$ is a map from $H^{n,i}(X)$ to 
$H^{n,i+d}(X)$; by Theorem \ref{NRT}, it is an isomorphism when $i\ge n$.  
By construction, this is the map in cohomology induced by
the change-of-truncation map
\begin{equation}\label{eq:c.o.truncation}
\F(i)\cong \tau^{\le i}R\alpha_*\muell{i}\to
           \tau^{\le i+d}R\alpha_*\muell{i} \cong \F(i+d).
\end{equation}
associated to the isomorphism of \'etale sheaves $\muell{i}\to\muell{i+d}$
sending the generator $\zeta^{\oo i}$ to the generator $\zeta^{\oo i+d}$.

Write $H^{n,i}(X)[1/b]$ for the colimit of
\[
H^{n,i}(X)\map{b}H^{n,i+d}(X)\map{b}\cdots\map{b} H^{n,i+jd}(X)\map{b}\cdots.
\]
From the diagram
\begin{equation*}
\begin{CD}
H^{n,i}(X) @>b>> H^{n,i+d}(X) @>b>> \cdots H^{n,i+jd}(X) @>b>> \cdots\\
@VVV  @VVV  @VVV \\  
H_\et^n(X,\muell{i}) @>{\cong}>> H_\et^n(X,\muell{i+d}) @>{\cong}>> \cdots 
H_\et^n(X,\muell{i+jd}) @>{\cong}>> \cdots
\end{CD}
\end{equation*}
we obtain a natural transformation from $H^{n,i}(X)[1/b]$ 
to $H_\et^n(X,\muell{i})$.

We can formulate this in the motivic derived category $DM$,
using the \'etale-to-Nisnevich change of topology map $\alpha$.
Recall from \cite[10.2]{MVW} that $H_\et^n(X,\muell{i})$ is isomorphic to
\[
\Hom_{DM_\et}(\alpha^*\Ftr X,\F(i)[n])\cong 
\Hom_{DM}(\Ftr X,R\alpha_*\muell{i}).
\]
The map $\F(i)=\tau^{\le i}R\alpha_*\muell{i}\to R\alpha_*\muell{i}$
is compatible with the map $\F(i)\to\F(i+d)\to R\alpha_*\muell{i+d}$,
so it factors through a map $\F(i)[1/b]\to R\alpha_*\muell{i}$,
where $\F(i)[1/b]$ denotes the (homotopy) colimit in $DM$ of 
\[
\F(i) \map{b} \F(i+d) \map{b} \F(i+2d) \map{b}\cdots\map{b} 
\F(i+jd) \map{b}\cdots.
\]

The following calculation is originally due to Levine \cite{Levine}.

\begin{thm}\label{H[1/b]} 
For each $i$, $\F(i)[1/b]\to R\alpha_*\muell{i}$
is an isomorphism in $DM$.

For $X$ smooth and all $n$, $H^{n,i}(X)[1/b]\map{} H_\et^n(X,\muell{i})$
is an isomorphism.
\end{thm}

\begin{proof}
Any complex $C$ is the homotopy colimit of the change-of-truncation maps
$\tau^{\le m}C\to\tau^{\le m+1}C$. For $C=R\alpha_*\muell{i}$, this
yields the first assertion.
The second assertion is an immediate consequence of this and the fact
that $\Ftr X$ is a compact object in $DM$, 
so $\Hom_{DM}(\Ftr X,-)$ commutes with homotopy colimits.
\end{proof}

Our next goal is to compare $P^0$ to the cohomology of the
change-of-truncation map 
$\tau^{\le i}R\alpha_*\muell{i}\to\tau^{\le i\ell}R\alpha_*\muell{i}$
of \eqref{eq:c.o.truncation}.

\begin{lem}\label{changetruncation}
The Frobenius map $\F(i)\map{\Phi}\F(i\ell)$ in motivic cohomology
is chain homotopic to the change-of-truncation map
\[ 
\F(i)\cong \tau^{\le i}R\alpha_*\muell{i}\to
           \tau^{\le i\ell}R\alpha_*\muell{i}\cong\F(i\ell).
\]
The Frobenius $H^{n,i}(X)\map{\Phi} H^{n,i\ell}(X)$ is multiplication by 
$b^{i(\ell-1)/d}=[\zeta^{\oo i(\ell-1)}]$.
\end{lem}

\begin{proof}
The Frobenius endomorphism is the identity on the \'etale sheaf of rings 
$\cO=\oplus_{i=1}^d\muell{i}$, so if we fix $i$ and an injective replacement
$\muell{i}\to I$, the Frobenius on $\muell{i}$ lifts to a map 
$f_i:I\to I$ which is chain homotopic to the identity. 
Since the product in motivic cohomology is induced from the product on 
$R\alpha_*\muell{i}=\alpha_*I$, the Frobenius in motivic cohomology is 
represented by the good truncation in degrees at most $i\ell$ of the composite
$\tau^{\le i}\alpha_*I\subset\alpha_*I \map{f_i} \alpha_*I$.
Since good truncation preserves chain homotopy, it is chain homotopic
to the canonical map $\tau^{\le i}\alpha_*I\subset\tau^{\le i\ell}\alpha_*I$.

The final assertion follows from \eqref{eq:c.o.truncation}.
\end{proof}

\begin{prop}\label{P0=Frob}
The map $P^0\!: H^{n,i}(X)\to H^{n,i\ell}(X)$ is 
multiplication by $b^{i(\ell-1)/d}\!.$ 
Equivalently, $P^0$ is the cohomology of the change-of-truncation map
\[\tau^{\le i}R\alpha_*\muell{i}\to\tau^{\le i\ell}R\alpha_*\muell{i}.\]
\end{prop}

\begin{proof}
Recall that the Godement resolution $\cF\to S^\mathdot(\cF)$ is a functorial
simplicial resolution of any sheaf $\cF$ by flasque sheaves. 
Letting $S_i^\mathdot$
denote the total complex of the Godement resolution of 
$\tau^{\le i}R\alpha_*\muell{i}$, it follows that the product on
$\oplus\tau^{\le i}R\alpha_*\muell{i}$ induces a product pairing 
$S_i^n\oo S_j^n\to S_{i+j}^n$ for all $n$. In particular,
the Frobenius on $R\alpha_*\muell{i}$ induces a map $S_i^n\to S_{i\ell}^n$.

In \cite[11.1]{Epstein}, Epstein shows that the Godement resolution 
satisfies the conditions of his section 8. 
By functoriality, the equivariant map 
$(R\alpha_*\muell{i})^{\oo\ell}\to R\alpha_*\muell{i\ell}$
constructed after Theorem \ref{NRT} lifts to an equivariant map
$(S_i^\mathdot)^{\oo\ell}\to S_{i\ell}^\mathdot$. This is the analogue of
\cite[8.3.2]{Epstein}, and is exactly what we need in order for
the proof of \cite[8.3.4]{Epstein} to work. Thus if we represent
$v\in H^{n,i}(X)$ by a cocycle $u$ in the algebra $H^0(X,S^nA)$, 
then $P^0v$ is represented by the element $u^\ell$ of $H^0(X,S^nA)$.
Therefore $P^0$ is represented by the Frobenius.
\end{proof}

\begin{subex}\label{P0(bx)}
Recall that $b\in H^{0,d}(k)$.  
Since $P^0(b)=b^\ell$ (by \ref{P0=Frob}), the 
Cartan formula \ref{Cartan} yields $P^a(b x)=b^\ell P^a(x)$.
\end{subex}

Recall from \cite[2.60]{V11} that a {\it split proper Tate motive}
is a direct sum of Tate motives $\mathbb L^i[j]$ with $j\ge0$. 
If the weights $i$ are at least $n$ then we say the motive has
weight $\ge n$. Note that the cohomology of $\mathbb L^i[j]$ is a 
free bigraded $H^{*,*}$-module with a generator in bidegree $(2i+j,i)$.

It follows that we have a K\"unneth formula (see \cite[4.1]{W09}): 
if $\Ftr(Y)$ is a split proper Tate motive then 
$H^{*,*}(Y)$ is a free bigraded $H^{*,*}$-module, and
\begin{equation}\label{eq:Kunneth}
H^{*,*}(X\times Y)\cong H^{*,*}(X)\oo_{H^{*,*}}H^{*,*}(Y).
\end{equation}

\begin{ex}
Let $K=K(\F(i),n)$ be the Eilenberg-Mac\,Lane space classifying 
$H^{n,i}(-,\F)$; if $n\ge2i\ge0$ then $M=\Ftr(K)$
is a split proper Tate motive of weight $\ge i$, by \cite[3.28]{V11}.
It follows that $H^{*,*}(K^{\oo p})$ is the $p$-fold
tensor product of $H^{*,*}(K)$ with itself over $H^{*,*}$.

Recall from \cite[3.1]{MVW} that $\F(i)[i]$ is represented by the
abelian presheaf $\Ftr(\mathbb G_m^{\wedge i})$ so $\F(i)[n]$ is 
represented by the simplicial abelian presheaf associated to 
$\Ftr(\mathbb G_m^{\wedge i})[n-i]$ when $n\ge i$. From the adjunction 
\[
\Hom_{\text{Hot}}(X,u\F(i)[n])\cong
\Hom_{DM}(\Ftr(X),\F(i)[n])=H^{n,i}(X),
\]
we see that the classifying space $K(\F(i),n)$ of $H^{n,i}(-,\F)$ is the
simplicial abelian presheaf $G=u\F(i)[n]$ underlying $\F(i)[n]$;
see \cite[p.\,5]{V11}.
\end{ex}

\begin{lem}\label{Bott-injects}
If $\Ftr(Y)$ is a split proper Tate motive then multiplication by $b^e$ 
is an injection from $H^{p,q}(Y,\F)$ into $H^{p,q+de}(Y,\F)$,
and hence into $H^{p,q}(Y,\F)[1/b]$.
\end{lem}

\begin{proof}
It suffices to consider $\Ftr(Y)=\mathbb L^i[j]$. There is no harm in
increasing $e$ so that $(\ell-1)|de$. 
Set $p'=p-2i-j$ and $q'=q-i$.  Since $H^{*,*}(Y)$ is a
free $H^{*,*}(k)$-module, the assertion for $H^{p,q}(Y)$
amounts to the assertion that 
either $H^{p',q'}(k)=0$ (and injectivity is obvious) or else 
$0\le p'\le q'$ and $H^{p',q'}(k)\cong H_\et^{p'}(k,\muell{q'})$.  
In the latter case, we also have 
$H^{p',q'+de}(k)\cong H_\et^{p'}(k,\muell{q'+de})$ 
and the isomorphism is induced from the isomorphism 
$\muell{q'}\cong\muell{q'+de}$.
\end{proof}

In the next Proposition, we write $K$ for $K(\F(i),n)$. For each $p$ and $q$,
there is a canonical map $H^{p,q}(K,\F)\to H_\et^p(K,\muell{q})$.
It sends the operations $P^a$ of Definition \ref{motivic-exist}
to the \'etale operations $P^a$ of Theorem \ref{etale-exist}.

\begin{prop}\label{prop:detecting}
If $n\ge2i$, the canonical map is an injection,
from the set $H^{p,q}(K,\F)$ of motivic cohomology operations 
$H^{n,i}\to H^{p,q}$ to the set $H_\et^p(K,\muell{q})$
of \'etale cohomology operations 
$H_\et^n(-,\muell{i})\to H_\et^p(-,\muell{q})$.
\end{prop}
%

\begin{proof}
By the usual transfer argument, we may assume that $\zeta\in k$.
Let $K$ denote the Eilenberg-Mac\,Lane space classifying $H^{n,i}(-,\F)$.
By \cite[3.28]{V11}, $\Ftr(K)$ is a split proper Tate motive.
By Lemma \ref{Bott-injects} and Levine's Theorem \ref{H[1/b]},
$H^{p,q}(K,\F)$ injects into 
$H^{p,q}(K,\F)[1/b]\cong H_\et^p(K,\muell{q})$.
Thus the group $H^{p,q}(K,\F)$ of motivic
cohomology operations injects into the group $H_\et^p(K,\muell{q})$
of \'etale cohomology operations.
\end{proof}

Recall that $d=[k(\zeta):k]$. By abuse of notation, if $d|m$ we write
$\zeta^m$ for the element $b^{m/d}$ of $H^{0,m}(k)$ defined 
at the start of this section.

\begin{cor}\label{PV-vs-P}
Suppose that $n\ge2i$ and $n\ge2a$. Then for $x\in H^{n,i}(X)$:
\begin{enumerate}
\item If $a\le i$, $P^a(x ) = [\zeta]^{(i-a)(\ell-1)} P_V^a(x)$;
\item If $a\ge i$, $P_V^a(x)= [\zeta]^{(a-i)(\ell-1)} P^a(x)$.
\end{enumerate}
\end{cor}

\begin{proof} (Cf.\ \cite[Thm.\,1.1]{BJ})
The two sides have the same bidegree, and agree with $P^a(x)$ 
in \'etale cohomology by Lemma \ref{lem:mot-et} and 
Remark \ref{PV-Pet}.
\end{proof}

\begin{cor}\label{PVn}
If $n\ge i$ and $x\in H^{2n,i}$ then 
$P_V^n(x)=[\zeta]^{(n-i)(\ell-1)}x^\ell$.
\end{cor}

\begin{proof}
This is the case $a=n$ of Corollary \ref{PV-vs-P}, 
as $P^n(x)=x^\ell$ (Proposition \ref{l-power}).
\end{proof}

\begin{thm}\label{Bock-motivic}
The motivic operations $Q^a$ on $H^{n,i}$ are related to the 
Bockstein $\beta$ by $Q^0(x)=b^{i(\ell-1)/d}\beta(x)$ 
and $Q^a=\beta P^a$ for $a>0$.
\end{thm}

\begin{proof}
Set $K=K(\F(i),n)$, so that motivic cohomology operations 
$H^{n,i}\to H^{p,q}$ correspond to elements of $H^{p,q}(K)$. In particular,
the identity on $H^{n,i}$ is represented by the canonical element $\iota$
of $H^{n,i}(K)$, and the motivic cohomology operation 
$Q^0-b^{i(\ell-1)/d}\beta$ is represented by the element 
$Q^0(\iota)-b^{i(\ell-1)/d}\beta(\iota)$ of 
$H^{n+1,i\ell}(K)$.
The map $H^{n+1,i\ell}(K)\to H_\et^{n+1}(K,\muell{i\ell})$
is an isomorphism if $n<i\ell$ by Theorem \ref{NRT}, and 
is an injection if $n\ge2i$ by Proposition \ref{prop:detecting}.
By Lemma \ref{lem:mot-et}, we have a commutative diagram
\[\begin{CD}
H^{n,i}(K) @>{Q^0-b^{i(\ell-1)/d}\beta}>> H^{n+1,i\ell}(K)  \\
@VVV   @V{\text{into}}VV \\
H_\et^{n,i}(K,\muell{i}) @>{0=Q^0-b^{i(\ell-1)/d}\beta}>> 
H_\et^{n+1}(K,\muell{i\ell}).
\end{CD}\]
The bottom map is zero by Lemma \ref{Bock-etale}.
It follows that $Q^0=b^{(\ell-1)/d}\beta$.
Similarly, if we set $N=n+a(\ell-1)+1$ then we have a diagram
\[\begin{CD}
H^{n,i}(K) @>{Q^a-\beta P^a}>> H^{N,i\ell}(K) 
@>{(P^0)^\nu}>{\text{into}}> H^{N,i\ell^{\nu+1}}(K)\\
@VVV  @VVV @V{\cong}VV \\
H_\et^{n,i}(K,\muell{i}) @>{0=Q^a-\beta P^a}>> H_\et^{N}(K,\muell{i\ell})
@>{\cong}>> H^{N}(K,\muell{i\ell^{\nu+1}}).
\end{CD}\]
The right vertical is an isomorphism by the Norm Residue Theorem \ref{NRT};
the upper right horizontal map is an injection by Proposition \ref{P0=Frob}
and Lemma \ref{Bott-injects},
and the lower left horizontal map is zero by Lemma \ref{Bock-etale}.
It follows that $Q^a=\beta P^a$.
\end{proof}

\medskip
\section{Borel's Theorem}
In order to go from $H^{1,*}$ to $H^{n,*}$, we need a slight
generalization of Borel's theorem \cite[6.21]{McCleary}, one which
accounts for the coefficient ring $H^{*,*}=H^{*,*}(\Spec k)$.

\begin{defn} 
Let $H^*$ be a graded-commutative $\F$-algebra. 
If $W^*$ is a graded $H^*$-algebra,
an {\it $\ell$-simple system of generators} of $W^*$ over $H^*$ is a
totally ordered set of elements $x_i$,
such that $W^*$ is a free left $H^*$-module on the monomials
$x_{i_1}^{m_1}\cdots x_{i_k}^{m_k}$, where the $i's$ are in order and
$0\le m_j<\ell$ (with $m_j\le1$ if $\deg(x_j)$ is odd).
\end{defn}

\begin{thm}\label{Borel}
Let $H^*$ be a graded-commutative $\F$-algebra with $H^0=\F$, and
suppose that $\{E_r^{*,*},d_r\}$ is a $1^{st}$-quadrant spectral sequence
of graded-commutative $H^*$-algebras converging to $H^*$. 
Set $V^*=E_2^{*,0}$ and $W^*=E_2^{0,*}$, and suppose that 
\\
(i) $E_2^{*,*}\cong W^*\otimes_{H^*}V^*$ as algebras, and that
(ii) the $H^*$-algebra $W^*$ has an $\ell$-simple system of generators
$\{ x_i\}$, each of which is transgressive.

Then $V^*$ is the tensor product of $H^*$ and
a free graded-commutative $\F$-algebra on generators
$y_i=\tau(x_i)$ and (when $\ell\ne2$ and $\deg(x_j)$ is even) 
$z_j=\tau(y_j\oo x_j^{\ell-1})$. (Here $\tau$ is the transgression.)
\end{thm}

\begin{proof}
The proof of Borel's Theorem in \cite[6.21]{McCleary} goes through.
\end{proof}
%


We use the bar construction to form the bisimplicial classifying 
spaces $B_\mathdot G$ (with $G^{p}$ in simplicial degree $p$) and
$E_\mathdot G$ (with $G^{p+1}$ in simplicial degree $p$). 
We write $\pi$ for the canonical projection $E_\mathdot G\to B_\mathdot G$.
The Leray spectral sequence becomes
\begin{equation}\label{eq:LSS}
E_2^{p,q} = H^p(B_\mathdot G,R^q\pi_*A) 
        \Rightarrow  H^{p+q}(E_\mathdot G,A).
\end{equation}

\begin{prop}\label{Leray}
Suppose that $G$ is a connected simplicial sheaf of groups on $T$ and 
$A$ is a sheaf of $\F$-algebras satisfying the K\"unneth condition that
\[
H^*(U,A)\otimes_{H^*(T,A)} H^*(G,A) \map{\cong} H^*(U\times G,A)
\]
is an isomorphism for all $U$. Then the Leray spectral sequence
\eqref{eq:LSS} satisfies condition (i) of Borel's Theorem
with $E_2^{p,q}=H^p(B_\mathdot G,A)\otimes_{H^*(T,A)} H^q(G,A)$.
\end{prop}

\begin{proof}
For simplicity of notation, let us write $\otimes_H$ for 
$\otimes_{H^*(T,A)}$. We first claim that the higher direct images
$R^q\pi_*(A)$ are $A\otimes_H H^*(G,A)$.
To see this, recall that $R^q\pi_*(A)$ is the sheafification of the
presheaf that to a map $U\to B_pG$ associates $H^q(\pi^{-1}U,A)$,
where $\pi^{-1}U=E_\mathdot G\times_{B_pG}U$ is  $U\times G$. 
By hypothesis,  $H^*(\pi^{-1}U,A)$ is 
$H^*(U,A)\otimes_H H^*(G,A)$. The claim follows, since 
sheafification commutes with $\otimes_H H^*(G,A)$ and the
sheaf associated to $H^q(-,A)$ is $A$ if $q=0$ and zero for $q>0$.

Thus we have $E_2^{p,q}=H^p(B_\mathdot G,A)\otimes_H H^q(G,A)$.
Since $H^0(B_\mathdot G,A)=H^0(T,A)$, we have
\[
E_2^{0,q}=H^0(B_\mathdot G,A)\otimes_H H^q(G,A)=H^q(G,A).
\]
Because $G$ is connected, $H^0(G,A)=H^0(T,A)$ and hence
$E_2^{p,0}=H^p(B_\mathdot G,A)$.
The fact that the spectral sequence is multiplicative follows from
the fact that $A$ is a sheaf of algebras, and the work of Massey \cite{Mas}.
\end{proof}

\begin{Kudosthm}\label{Kudo}
Suppose $G$ and $A$ satisfy the hypotheses of Proposition \ref{Leray}.
If $x\in H^n(G,A)$ transgresses to $y\in H^{n+1}(B_\mathdot G,A)$ then
\begin{enumerate}
\item $\beta(x)$ transgresses to $-\beta(y)$;
\item $P^a(x)$ transgresses to $P^a(y)$; and
\item if $n=2a$ then $x^{\ell-1}\oo y$ transgresses to $-\,Q^a(y)$.
\end{enumerate}
\end{Kudosthm}

Any simplicially stable operation commutes with the transgression;
see \cite[6.5]{McCleary}.  Hence part (2) of
Theorem \ref{Kudo} is immediate whenever we know that
$P^a$ is simplicially stable. This is so for the operations $P^a$
in \'etale and motivic cohomology (by \ref{S-stable} and \ref{S1-stable}).

\begin{proof}  (Cf.\ \cite[3.4]{May})
As in the proof of Theorem \ref{Kunneth}, we fix a quasi-isomorphism
$A\map{\sim} I^*$. Let $f=\pi^*$ and $g=i^*$ be the canonical maps 
$I(G)\lmap{g}I(E_\mathdot G)\lmap{f}I(B_\mathdot G)$
coming from $G\map{i}E_\mathdot G\map{\pi} B_\mathdot G$.
The assertion that $x$ transgresses to $y$ 
means that there is a cocycle $b$ in $I^{n+1}(B_\mathdot G)$ 
representing $y$, and an element $u$ in $I^n(E_\mathdot G)$, 
such that $f(b)=du$ and $g(u)$ is a cocyle representing $x$,

Since the Bockstein satisfies $g(\beta u)=\beta g(u)$ and
$f(\beta b)=\beta(du)=-d(\beta u)$, we see that
$\beta(x)$, which is represented by $g(\beta u)$, 
transgresses to $-\beta(y)$.

Recall from Section \ref{sec:theta} that $b$ and $u$ determine 
a cocyle $P^a(b)$ in $I^*(B_\mathdot G)$ representing $P^a(y)$
and an element $P^a(u)$ in $I^*(E_\mathdot G)$ so that
$P^a(x)$ is represented by $P^a g(u)=gP^a(u)$.
By Lemma \ref{dP-Pd}, we have
\[
f P^a(b) = P^a f(b) = P^a(du) = dP^a(u).
\]
It follows that $P^a(x)$ trangresses to $P^a(y)$.

Since $b$ is a cocycle, $Q^a(b)$ represents $Q^a(y)$, and 
by Lemma \ref{dP-Pd} we have
\[
f Q^a(b)=Q^a f(b) = Q^a(du) = -\, d(Q^a u). 
\]
Thus the class of $Q^a(u)$ transgresses to $-\,Q^a(y)$, and
it suffices to show that $Q^a(u)$ 
represents $x^{\ell-1}\otimes y$ under the isomorphism
$E_2^{p,q}\cong H^p(B_\mathdot G)\oo H^q(G)$ of \ref{Leray}. 

Recall from \eqref{eq:P=D} that $\nu_n=(-1)^r m!^{-n}$, where $m=(\ell-1)/2$.  
We have $\nu_n=(-1)^a$, because $n=2a$, 
$(m!)^2=(-1)^{m+1}$ and $r\equiv am\pmod2$. 
We now follow p.\,167 of \cite{May} up to (9). 
Starting from $u\in I^n(X)$, 
May produces elements $t_i$ in $I^{\oo\ell}(X)$
and a family of elements $\{ c_a\}$, $\{ c'_a\}$ in
$C_*\oo I^{\oo\ell}(X)$, depending naturally on $u$, such that 
\[
Q_M^a(u)=(-1)^{a}\nu(1-n)\;\theta(c'_a)=m!\;\theta(c'_a).
\]
The analysis of the terms in $c'_a$ on top of p.\,171 of \cite{May}
shows that there is a term $c''$ such that $c'-d(c'')$ is $(-1)^m m!\;z$
plus terms mapped by $\theta$ into lower parts of the filtration,
where $z= e_0\oo u\oo \cdots \oo u \oo du$, and that
$\theta(z)$ represents $x^{\ell-1}\oo y$. Therefore, up to terms in
lower parts of the filtration we have
$Q_M^a(u)=m!\;\theta(c'_a)=(-1)^m(m!)^2\theta(z)=-\theta(z)$.
Since we saw in Corollary \ref{PQagree} that
$Q^a(u) = Q_M^a(u)$,  
the result follows.
\end{proof}

\medskip\goodbreak

We illustrate the use of Proposition \ref{Leray} with the \'etale topology.
First, consider the \'etale sheaf $G=\mu_\ell$. 
If $\mu_\ell$ is connected then it does not satisfy the 
K\"unneth condition of Proposition \ref{Leray} for $U=\Spec(\bar k)$.
Indeed, $H^0(\mu_\ell,\F)=\F$ yet
$H^0(G\times\Spec{\bar k},\F)=\prod_1^\ell \F$. 
%
However, things change if we consider the \'etale sheaf 
$\cO=\bigoplus_{i=0}^{d-1} \muell{i}$ of Section \ref{sec:Aetale}.

\goodbreak
\begin{lem}\label{Leray-etale}
$H_\et^*(X\times\muell{i},\cO)\cong 
H_\et^*(X,\cO)\otimes_{H^*(k,\cO)} H_\et^*(\muell{i},\cO).$
\end{lem}

\begin{proof}
As an \'etale sheaf of $\F$-modules, constant over $k(\zeta)$,
$\F[\muell{i}]$ is a direct sum of the locally constant sheaves 
$\muell{\alpha}\!,$ each of which is an invertible object.
Because $\F[X\times\muell{i}]\cong \F[X]\otimes\F[\muell{i}]$,
$H_\et^n(X\times\muell{i},\muell{q})$ equals
\begin{align*}
\Ext^n(\F[X]\otimes\F[\muell{i}],\muell{q})\cong&\
\Ext^n(\F[X],\mathcal{RH}om(\F[\muell{i}],\muell{q})) \\ 
\cong&\ \Ext^n(\F[X],\mathcal{RH}om(\F,\oplus\muell{q-\alpha}))\\
\cong& \oplus_\alpha \Ext^n(\F[X],\muell{q-\alpha})
\cong \oplus_\alpha H_\et^n(X, \muell{q-\alpha}).
\end{align*}
The pairing $H_\et^*(X,\cO)\oo_{\F} H_\et^*(\muell{i},\cO)\to
H_\et^*(X\times\muell{i},\cO)$ is the direct sum over $\alpha$, 
$s$ and $t$ of the top row in the commutative diagram
\[\begin{CD}
\Ext^*(\F[X],\muell{s})\oo_{\F}\Ext^*(\muell{\alpha},\muell{t}) @>>>
\Ext^*(\F[X]\oo\muell{\alpha},\muell{s+t}) \\
@V{\cong}VV @VV{\cong}V\\
\Ext^*(\F[X],\muell{s})\oo_{\F}\Ext^*(\F,\muell{t-\alpha})  @>>>
\Ext^*(\F[X],\muell{s+t-\alpha}).
\end{CD}\]
Since $H_\et^*(k,\cO)=\Ext^*(\F,\cO)\cong\oplus_t\Ext(\F,\muell{t-\alpha})$
for each $\alpha$, setting $s=q-\alpha$ and
summing over $s$ and $t$ yields the result.
\end{proof}


\begin{cor}\label{LSSetale}
If $Y$ is a coproduct of schemes which are finite products of $\muell{i}$, then
\[ 
H_\et^*(X\times Y,\cO) \cong H_\et^*(X,\cO)\oo_{H^*(k,\cO)} H_\et^*(Y,\cO).
\]
\end{cor}


\begin{ex}\label{Het2-ops}
The ring of all \'etale cohomology operations from 
$H_\et^{2}(-,\muell{i})$ to $H_\et^*(-,\muell{*})$
is the free left $H_\et^{*}(k,\cO)$-module on generators in 
$H^{*,i}(K_2)$: monomials in the identity
($\text{id}\in H_\et^2(K_2,\muell{i})$), $\beta$, 
the $P^I\beta$ and the $\beta P^I\beta$
($P^I=P^{\ell^\nu}\!\cdots P^\ell P^1$).
This result, proven in Theorem \ref{thm:etaleops} above, 
can also be obtained from the Leray spectral sequence \eqref{eq:LSS}.

Each term in the simplicial sheaf $B_\mathdot \muell{i}$ is a 
coproduct of products of $\muell{i}$, so Corollary \ref{LSSetale}
and Proposition \ref{Leray} imply that the Leray
spectral sequence satisfies condition (i) of Borel's Theorem \ref{Borel}.
The explicit description of  $H_\et^*(B\muell{i},\cO)$ in
Proposition \ref{Het1-ops} as $H_\et^*(k,\cO)\oo\F[u,v]/(u^2)$, 
shows that it has an $\ell$-simple system of 
generators: $u$, and the $x_\nu=v^{\ell^\nu}$ for $\nu\ge0$.
The transgression $\tau$ sends $u$ to $\iota$, so $v=\beta u$ 
transgresses to $-\beta$, by Kudo's Theorem \ref{Kudo}(1).
Thus condition (ii) is also satisfied, and Borel's Theorem states
that $H_\et^*(K_2,\cO)$ is the free 
graded-commutative $H^{*,*}$-algebra on generators 
$\iota\in H^{2,i}(K_2)$,
$$
y_\nu=\tau(x_\nu)\in H^{2\ell^\nu+1}(K_2,\muell{i\ell^\nu})
\quad\text{and}\quad 
z_\nu=\tau(x_\nu^{\ell-1}\oo y_\nu)\in 
H^{2\ell^{\nu+1}\!+2}(K_2,\muell{i\ell^{\nu+1}}).
$$
Note that $y_0=\beta(\iota)$. 
Since $x_{\nu+1}=x_\nu^{\ell^\nu}=P^{\ell^\nu}x_\nu$, 
Kudo's Theorem \ref{Kudo}(2) and an inductive argument show that
$y_{\nu+1}$ is $P^{\ell^\nu}y_\nu$
and also $P^{\ell^\nu}\!\cdots P^\ell P^1\beta$.
This completes the proof for $\ell=2$.

For $\ell>2$, it remains to show that $-z_\nu$ is 
$\beta P^{\ell^\nu}(y_\nu)=\beta P^{\ell^\nu}\!\cdots P^\ell P^1\beta$.
This follows from Kudo's Theorem \ref{Kudo}(3) and Lemma \ref{Bock-etale}.
\end{ex}

\bigskip
\section{Motivic operations on weight~1 cohomology}\label{sec:weight1}

We now turn to natural operations defined on the motivic cohomology groups
with weight~1, \ie $H^{n,1}(X)=H^{n,1}(X,\F)$. 
We begin with the case $n=1$.

Let $\mu_\ell$ be the group scheme of $\ell^{th}$ roots of unity.
On pp.\,130--131 of \cite{MV}, Morel and Voevodsky define a simplicial
Nisnevich sheaf $B_{et}\mu_\ell$ and observe that it classifies
the \'etale cohomology group $H_\et^1(-,\mu_\ell)$, and hence the 
motivic group $H^{1,1}$ by Theorem \ref{NRT},
in the sense that $[X_+,B_{et}\mu_\ell]\cong H^{1,1}(X)$ for every
smooth simplicial scheme $X$ over $k$.

Following \cite[p.\,17]{RPO}, we write $B\mu_\ell$ for the
geometric classifying space of $\mu_\ell$, constructed in 
\cite[p.\,133]{MV} 
(where the notation $B_{gm}\mu_\ell$ was used).
By \cite[4.2.7]{MV}, $B\mu_\ell$ is $\mathbb{A}^1$-equivalent to 
$B_{et}\mu_\ell$, so it also classifies $H^{1,1}$.

When $\ell=2$, the generator $[\zeta]$ of $H^{0,1}(k)=\mu_2(k)$ and 
its Bockstein, the element
$[-1]\in H^{1,1}(k)=k^\times/k^{\times\ell}$, 
play an important role.

\begin{prop}\label{prop:Bmu}
There are elements $u\in H^{1,1}(B\mu_\ell)$, $v\in H^{2,1}(B\mu_\ell)$ 
such that
\[ H^{*,*}(B\mu_\ell) \cong \begin{cases}
H^{*,*}(k)\oo \F[u,v]/(u^2), &\ell\ne2 \\
 H^{*,*}(k)\oo \F[u,v]/(u^2+[-1]u+[\zeta ] v), &\ell=2.
\end{cases}\]
Thus every cohomology operation on $H^{1,1}(X)$ is uniquely a sum of
the operations $x\mapsto cx^\varepsilon\beta(x)^m$, 
where $c\in H^{s,j}(k)$, $m\ge0$ and $0\le\varepsilon\le1$.
\end{prop}

\begin{proof}
This is the special case $F_\mathdot=S^0$ in Proposition 6.10 of \cite{RPO}.
Note that the operation $cx^\varepsilon(\beta x)^m$ has bidegree
$(s+2m+\varepsilon-1,j+m+\varepsilon-1)$.
\end{proof}

As in Example \ref{Het2-ops}, we can use this as the starting point
to describe all motivic operations on $H^{*,1}$. 
For example, the motive of $G=B\mu_\ell$ is 
a split proper Tate motive, so \eqref{eq:Kunneth} holds. 
\edit{tweaked 12/22}
By Proposition \ref{Leray} the 
Leray spectral sequence has the form
\begin{equation}\label{mss:p=2}
E^{p,q}_2 = H^{p,*}(B_\mathdot B\mu_\ell) \oo_{H^{*,*}(k)} H^{q,*}(B\mu_\ell)
\Rightarrow H^{p+q,*}(E_\mathdot B\mu_\ell) \cong H^{p+q,*}(k).
\end{equation}

\begin{cor}\label{H21-ops}
If $\ell\ne2$, the ring of cohomology operations on $H^{2,1}$ is 
the tensor product of $H^{*,*}(k,\F)$ and the
free graded-commutative algebra generated by: 
the identity of $H^{2,1}$, the Bockstein $\beta$, 
the $P^I\beta$ 
and the $\beta P^I\beta$ 
where $P^I=P^{\ell^\nu}\!\cdots P^\ell P^1$.

For $\ell=2$, the ring of cohomology operations on $H^{2,1}$ is 
the tensor product of $H^{*,*}(k,\mathbb F_2)$ and the
free graded-commutative algebra generated by: 
the identity of $H^{2,1}$, $Sq^1$, ..., $Sq^I=Sq_V^I$, where
$Sq^I=Sq^{2^\nu}\!\cdots Sq^2 Sq^1$.
\end{cor}

\begin{proof}
Since the simplicial scheme $B_\mathdot B\mu_\ell$ is 
$K(\F(1),2)$,
we merely need to compute its motivic cohomology using \eqref{mss:p=2}.

By Proposition \ref{prop:Bmu},
$H^{*,*}(B\mu_\ell)$ has an $\ell$-simple system of generators over
$H^{*,*}$ consisting of $u$ and the $x_\nu=v^{\ell^\nu}$ for $\nu\ge0$.
Since \eqref{eq:Kunneth} holds, Proposition \ref{Leray} implies
that the hypotheses of Borel's Theorem \ref{Borel} and
Kudo's Theorem \ref{Kudo} hold for \eqref{mss:p=2}.
Therefore $H^{*,*}(B_\mathdot B\mu_\ell)$ is the tensor product of
$H^{*,*}$ and the free graded-commutative $\F$-algebra on generators
$\iota$, $y_\nu=\tau(x_\nu)$ --- and $z_\nu$ if $\ell>2$. 
Since $u$ transgresses to $\iota$, $x_0=v=\beta(u)$
and $x_{\nu+1}=P^{\ell^\nu}x_\nu$, Kudo's Theorem \ref{Kudo} 
implies (by induction) that $y_0=\beta(\iota)$,
$y_{\nu+1}=P^{\ell^\nu}y_\nu=P^I\beta(\iota)$ and (using
Theorem \ref{Bock-motivic}) that $z_\nu$ is 
$-\beta P^I\beta(\iota)$.
\end{proof}
\goodbreak

\smallskip

To describe cohomology operations on $H^{n,1}$, we use the algebra
$H_\top^*(K_n)$, defined in \ref{def:PI}. We bigrade it by giving 
it the weight grading that $P^I$ has weight $\ell^k-1$, where
$I=(\epsilon_0,s_1,\epsilon_1,...,s_k,\epsilon_k)$.

\begin{thm}\label{Hn1-ops}
For each $n\ge1$, the ring of all motivic cohomology operations on
$H^{n,1}$ is isomorphic to the free left $H^{*,*}$-module
$H^{*,*}(k)\otimes H^*_\top(K_n)$ in which the $P^I$ are bigraded
according to Definition \ref{motivic-exist}.

Thus every cohomology operation on $H^{n,1}(X)$ is a sum of the operations
$x\mapsto c(P^{I_1}x)(P^{I_2}x)\cdots(P^{I_s}x)$, where $c\in H^{*,*}(k)$
and each $I_j$ satisfies the excess condition of Definition \ref{def:PI}.
\end{thm}

\begin{proof}
We proceed by induction on $n$, the cases $n=1,2$ being given above.
Set $K_{n}=K(\F(1),n)$, so $K_{n+1}=B_\mathdot(K_n)$, 
\edit{proof tweaked 12/22}
and suppose inductively that
the algebra $H^{*,*}(K_n)$ is given as described in the theorem,
so that it has an $\ell$-simple system of generators consisting of the
$P^I(\iota_n)$ with $I$ admissible and $e(I)<n$ 
(or $e(I)=n$ and $\epsilon_1=1$),  
and $\ell^\nu$ powers of the $P^I(\iota_n)$ of even degree. 

Since $\Ftr(K_n)$ is a split proper Tate motive by \cite[3.28]{V11},
the K\"unneth condition \eqref{eq:Kunneth} of Proposition \ref{Leray}
holds. Hence the 
hypotheses of Borel's Theorem \ref{Borel} are satisfied and the
Leray spectral sequence \eqref{eq:LSS} has the form
\begin{equation*}
E^{p,q}_2 = H^{p,*}(K_{n+1}) \oo_{H^{*,*}(k)} H^{q,*}(K_n)
\Rightarrow H^{p+q,*}(E) \cong H^{p+q,*}(k).
\end{equation*}
Therefore 
$H^{*,*}(K_{n+1})$ is the tensor product of $H^{*,*}$ and
a free graded-commutative $\F$-algebra on certain generators; 
it remains to establish that
they are the ones describe in the theorem. But, except for weight
considerations, this is exactly the same as in the topological case,
as presented on p.\,200 of \cite{McCleary}. Of course, the weight of the
$x_I=P^I(\iota_n)$ is the same as the weight of $y_I=P^I(\iota_{n+1})$. 
Inspection of the weights of the new generators 
$P^{\ell^t s}\cdots P^s y_I$ (when $x_I$ has degree $2s$)
shows that each additional $P^{\ell^t s}$ multiplies the weight
by $\ell$, as required.
\end{proof}

\newpage
\section{Motivic operations on degree 1 cohomology}\label{sec:degree1}

We now turn to operations defined on $H^{1,*}$. Here we encounter
new cohomology operations arising from the Norm Residue Theorem \ref{NRT},
representing a negative twist. Here are a couple of examples.

\begin{ex}\label{ex:eta}
There are operations $H^{1,r(\ell-1)}(X)\map{\sim} H^{1,1}(X)$,
since both groups are naturally isomorphic to $H_\et^1(X,\mu_\ell)$.
An element $\eta\in H_\et^1(k,\muell{2-i})$ determines a natural 
transformation $H^{1,i}(X)\to H^{2,2}(X)$. 
\end{ex}

\paragraph{\bf The case $k=k(\zeta)$}
If $k$ contains a primitive $\ell^{th}$ root of unity $\zeta$, 
the classification is immediate from Proposition \ref{prop:Bmu}.
Let $[\zeta]$ be the class of $\zeta$ in $H^{0,1}(k)\cong\mu_\ell$.

\begin{prop}\label{H1i_withzeta}
Suppose that $\zeta\in k$ and $i>1$. Then there is a natural isomorphism
$\gamma:H^{1,i}(X)\map{\sim} H^{1,1}(X)$, and $[\zeta]^{i-1}\cup\gamma(x)=x$.

Every motivic cohomology operation on $H^{1,i}$ is uniquely a sum of the 
operations $x\mapsto c(\gamma x)^\varepsilon\beta(\gamma x)^m$, 
where $c\in H^{*,*}(k)$, $0\le\varepsilon\le1$ and $m\ge0$.
\end{prop}

\begin{proof}
By Theorem \ref{NRT}, $H^{1,i}(X)\cong H_\et^1(X,\muell{i})$ for all $i>0$.
Since multiplication by $[\zeta^{i-1}]$ is an isomorphism between
$H_\et^n(X,\mu_\ell)$ and $H_\et^n(X,\muell{i})$, 
its inverse isomorphism $\gamma$ is natural.
Via $\gamma$, operations on $H^{1,i}$ correspond
to operations on $H^{1,1}$, which are described in Proposition \ref{prop:Bmu}.
\end{proof}

For example if $i\ge2$ and $\eta\in H_\et^1(k,\muell{2-i})$ then 
the operation $H^{1,i}\to H^{2,2}$ of Example \ref{ex:eta} is the 
operation $x\mapsto c(\gamma x)$ of Proposition \ref{H1i_withzeta}, 
where $c=\eta\cup[\zeta]^{i-1}$.

\begin{subrem}
If $c\in H^{s,j}(k)$ then $\phi(x)= c(\gamma x)^\varepsilon(\beta\gamma x)^m$
is a cohomology operation of bidegree 
$(s+\varepsilon+2m-1,j+\varepsilon+m-i)$.
In particular $\gamma$ is a cohomology operation of bidegree $(0,1-i)$.
\end{subrem}

\medskip\goodbreak
\paragraph{\bf Galois descent}
We now consider the situation in which $\mu_\ell\not\subset k$.
Clearly, not all cohomology operations defined over $k(\zeta)$
are defined over $k$. However, some of these operations do descend,
such as those in Example \ref{ex:eta}.


It is convenient to consider the \'etale cohomology
of $k$ as being bigraded, by integers $n\ge0$ and $i\in\Z$,
with $H_\et^n(k,\muell{i})$ in bidegree $(n,i)$. Thus the
motivic cohomology ring $H^{*,*}(k)$ is a bigraded subring of 
$H_\et^*(k,\muell{*})$. 

\begin{defn}
For each integer $b$, let $\zeta^{-b}H^{*,*}(k)$ denote the
direct sum of all $H_\et^s(k,\muell{t})$ with $0\le s\le t+b$. This
is a bigraded $H^{*,*}(k)$-submodule of $H_\et^*(k,\muell{*})$.
It is a cyclic module if and only if $\zeta^b\in k$, when it is
the $H^{*,*}(k)$-submodule generated by $[\zeta^b]\in H_\et^0(k,\muell{-b})$.
\end{defn}

\begin{thm}
Fix an integer $i\ge2$. Then the ring of cohomology operations on
$H^{1,i}$ is the direct sum of copies of $\zeta^{-b}H^{*,*}(k)$, 
$b=(i-1)(\varepsilon+m)$, over integers $m\ge0$ and $\varepsilon\in\{0,1\}$
If $0\le s\le t+b$, the operation corresponding to 
$c\in H_\et^s(k,\muell{t})$, 
$m$ and $\varepsilon$ sends $H^{1,i}(X)$ to 
$H^{s+\varepsilon+2m,t+b+\varepsilon+m}(X)$:
\[
\phi(\zeta^{i-1}\cup y) = (\zeta^b\cup c)y^\varepsilon\beta(y)^m .
\]
\end{thm}

\begin{proof}
Let $G$ denote the Galois group of $k(\zeta)/k$.
Since $H^{*,*}(X)$ is the $G$-invariant summand of $H^{*,*}(X(\zeta))$,
a motivic operation $H^{1,i}(X)\to H^{*,*}(X)$ is the same thing as a
$G$-invariant operation $H^{1,i}(X)\to H^{*,*}(X(\zeta))$.
Given $x\in H^{1,i}(X)$, there is a unique $y\in H^{1,1}(X(\zeta))$
so that $x=[\zeta]^{i-1}\cup y$, where $[\zeta]\in H^{0,1}(k(\zeta))$. 
By Proposition \ref{H1i_withzeta},
we are reduced to determining when $G$ acts trivially on
$c'y^\varepsilon(\beta y)^m$. Since $y^\varepsilon(\beta y)^m$ is in the
summand of $H^{*,*}(k(\zeta))$ which is isotypical for $\muell{-b}$,
this holds if and only if $c'$ is in the summand of $H^{s,j}(k(\zeta))$ 
which is isotypical for $\muell{b}$. By \eqref{eq:Shapiro}, there is a
unique $c\in H^{s,j-b}(k)$ so that $c'=[\zeta]^{b}\cup c$.
\end{proof}

\begin{ex}[$b=1$]\label{ex:b=1}
An element $c$ in $H_\et^1(k,\F)=\Hom(\text{Gal}(\bar{k}/k),\F)$ 
determines operations
$C:H^{1,2}(X)\to H^{2,1}(X)$ and $\phi: H^{1,2}(X)\to H^{3,2}(X)$.
If $y\in H^{1,1}(X(\zeta))$ is such that $x=[\zeta]\cup y$ then,
regarding $\zeta c$ as an element of $H^{1,1}(k(\zeta))$, we have
$C(x)=(\zeta c)y$ and $\phi(x)=(\zeta c)\beta(y)$. Of course, we can
identify $C$ with the map $H^1_\et(X,\mu_\ell)\map{\cup c}H^2_\et(X,\mu_\ell)$.

An element $t$ in $H_\et^2(k,\mu_\ell)$ (the $\ell$-torsion subgroup of the
Brauer group of $k$) determines operations $H^{1,2}(X)\to H^{3,3}(X)$
and $H^{1,2}(X)\to H^{4,3}(X)$. 
\edit{$H^{43}$ 12/31}
Writing $x=[\zeta]\cup y$ in $H^{1,2}(X(\zeta))$, the operations 
followed by the inclusion $H^{*,*}(X)\subset H^{*,*}(X(\zeta))$ send
$x$ to $([\zeta]\cup t)y$ and  $([\zeta]\cup t)\beta(y)$, respectively.
As mentioned in the introduction, we can identify the first operation
with $H^1_\et(X,\muell{2})\map{\cup t}H^3_\et(X,\muell{3})$.
\end{ex}


\medskip\goodbreak
\section{Conjectural matter}\label{sec:conjectures}
\medskip

In the preceeding two sections we have classified 
motivic cohomology operations on $H^{n,i}$ when $n=1$ or $i=1$.
We have also classified operations whose targets lie  inside the
``\'etale zone'' where $n\le i$. We know little about the intermediate
zone where $i<n<2i$. In this section we make some guesses about 
operations in the ``topological zone'' where $n\ge2i$.

\begin{ex}
There are many operations defined on $H^{n,2}$, $n\ge2$.
Let us compare Voevodsky's operation $P^1_V$
(landing in $H^{n+2\ell-2,\ell+1}$) with our operation $P^1$ 
(landing in $H^{n+2\ell-2,2\ell}$). Thus $P^1$ has the same bidegree
as $[\zeta]^{\ell-1}P^1_V$, where $[\zeta]\in H^{0,1}(k)$.
If $n\ge4$, we have $P^1=[\zeta]^{\ell-1}P^1_V$
by Corollary \ref{PV-vs-P}. If $n=2$ we also have
$P^1=[\zeta]^{\ell-1}P^1_V$ because
they induce the same \'etale operation ($P^1$) from
$H^{2,2}(X)\cong H_\et^2(X,\muell2)$ to
$H^{2\ell,2\ell}(X)\cong H_\et^{2\ell}(X,\muell{2\ell})$.
We do not know if $P^1$ and $[\zeta]^{\ell-1}P^1_V$ agree on $H^{3,2}$.
%
\end{ex}

Suppose that $\phi$ is a motivic cohomology operation on $H^{n,i}$
where $n\ge2i$. Passing to \'etale cohomology sends $\phi$ to an
\'etale operation, which by Theorem \ref{thm:etaleops} is a polynomial
in the \'etale operations $P^I$.
By Proposition \ref{prop:detecting}, some multiple of the Bott element
$b$ sends $\phi$ to operations $b^N\phi$ which are in the subalgebra
generated by the motivic operations $P^I$ defined in \ref{motivic-exist}.
It remains to determine what those powers are.

The following result of Voevodsky \cite[3.6--7]{RPO} shows
that all non-trivial operations in the topological zone increase $n$.

\begin{lem}\label{iota}[Voevodsky]
There are no motivic cohomology operations from $H^{2i,i}$ to 
$H^{n,j}$ when $j<i$, or when $i=j$ and $(n,j)\ne(2i,i)$. 
The module of motivic cohomology operations from $H^{2i,i}$ to 
$H^{*,i}$ is isomorphic to $\F$, on the identity.
\end{lem}

%

\begin{Conjecture}\label{Conj}
Assume that $k$ contains all primitive $\ell^{th}$ roots of unity,
and that $n\ge2i$.
Then the module of all motivic cohomology operations on $H^{n,i}(-,\F)$
is the tensor product of $H^{*,*}$ and a free graded polynomial algebra 
over $\F$ with generators all
$P^IP_V^J$, where $I=(\epsilon_0,s_1,\epsilon_1,...,s_k,\epsilon_k)$,
$J=(s_{k+1},\epsilon_{k+1},...,s_m,\epsilon_m)$ subject to the conditions
that (a) the concatenation $IJ$ is admissible with excess $e(IJ)$ either
$<4$ or else $\epsilon_0=1$ and $e(IJ)=4$; and
(b) for all $j>k$, $s_j<i+(\ell-1)\sum_{j+1}^m s_i$.
\end{Conjecture}

For $(n,i)=(4,2)$ this conjecture implies that among the polynomial 
generators for the motivic operations on $H^{4,2}$ we find
$P^{\ell^2+\ell+1}\beta P_V^{\ell+1}\beta P_V^1\beta$.
If $\ell=2$, we may rewrite this operations as
$Sq^{14} Sq_V^7 Sq_V^3 Sq_V^1$; compare with \cite[3.57]{V11}.
%

\begin{lem}\label{conj-induct}
If Conjecture \ref{Conj} holds for $H^{2i,i}$ then 
it holds for all $H^{n,i}$ with $n\ge2i$.
\end{lem}

\begin{proof}
We consider the Leray spectral sequence \eqref{eq:LSS} for
$G=K(\F(i),n)$ and $K=B_\mathdot G=K(\F(i),n+1)$ when $n\ge2i$. 
By induction, $H^{*,*}(G)$ is a polynomial
algebra over $H^{*,*}$ with an $\ell$-simple system $\{ x_i\}$ of generators.
By \cite[3.28]{V11}, $\Ftr(G)$ is a split proper Tate motive, 
so the K\"unneth condition of Proposition \ref{Leray} holds,
and Borel's Theorem \ref{Borel} implies that $H^{*,*}(K)$
is the tensor product of $H^{*,*}$ and a free graded-commutative 
$\F$-algebra on generators $y_i=\tau(x_i)$ and, 
when $\deg(x_j)$ is even and $\ell>2$,
$z_j=\tau(x_j^{\ell-1}\oo y_j)$. 

We now use the fact that the transgression commutes with any 
($S^1$-)stable cohomology operation, such as $P_V^J$; 
see \cite[6.5]{McCleary}.
Since the tautological element $\iota_n$ of $H^{n,i}(G)$ transgresses 
to the tautological element $\iota_{n+1}$ of $H^{n+1,i}(K)$, 
the generator $x_j=P^IP_V^J(\iota_n)$ transgresses
to $y_j=P^IP_V^J(\iota_{n+1})$ by Kudo's Theorem \ref{Kudo}.
This finishes the proof for $\ell=2$.

If $\ell$ is odd and $x_j=P^IP_V^J(\iota_n)$ has degree $2a$,
\edit{expanded 12/22/12}
the transgression $z_j$ of $x_j^{\ell-1}\oo y_j$ is 
$-\beta P^a P^IP_V^J(\iota_{n+1})$ by Kudo's Theorem \ref{Kudo}(3).
This finishes the proof for $\ell$ odd.
\end{proof}

\bigskip

\section*{Acknowledgements}

The authors would like to thank Peter May for his encouragement,
and for providing us with a copy of \cite{SErr}.

\bigskip

\end{document}